\documentclass[11pt, letter]{article}
\usepackage[letterpaper, left=1.truein, right=.8truein, top = 1.truein, bottom = 1.truein]{geometry}

\usepackage{eulervm}
\usepackage{tgpagella}
\usepackage[T1]{fontenc}
\usepackage{amsmath}
\usepackage{amssymb}
\usepackage{amsthm}
\usepackage{cite}
\usepackage{color}
\usepackage{graphicx}
\usepackage{color}
\usepackage{bm}
\usepackage{mathtools}

\DeclareMathAlphabet{\mathpzc}{OT1}{pzc}{m}{it}

\DeclareMathOperator{\rank}{rk}

\usepackage{algorithm}
\usepackage{algpseudocode}
\makeatletter \def\BState{\State\hskip-\ALG@thistlm} \makeatother

\usepackage{etoolbox} \makeatletter
\patchcmd{\@maketitle}{\begin{center}}{\begin{flushleft}}{}{}
\patchcmd{\@maketitle}{\begin{tabular}[t]{c}}{\begin{tabular}[t]{@{}l}}{}{}
\patchcmd{\@maketitle}{\end{center}}{\end{flushleft}}{}{}

\newenvironment{example}[1][Example]{\begin{trivlist}\item[\hskip \labelsep {\bfseries #1}]}{\end{trivlist}}

\newcommand\be{\begin{equation}}
\newcommand\ee{\end{equation}}

\newtheorem{thm}{Theorem}[section]
\newtheorem{lem}[thm]{Lemma}
\newtheorem{prp}[thm]{Proposition}

\newtheorem{rem}[thm]{Remark}
\newtheorem{obs}[thm]{Observation}

\newtheorem{ass}[thm]{Assumption}

\newcommand{\C}{\mathbb{C}}

\newcommand{\bP}{\mathbb{P}}
\newcommand{\N}{\mathbb{N}}

\newcommand{\Sym}[1]{Sym_{#1}(\R)}

\newcommand{\supp}{\text{supp}}

\newcommand\path{\bm{\pi}}

\DeclarePairedDelimiterX\set[1]\lbrace\rbrace{#1}

\newcommand{\critCurve}{\ensuremath{\Sigma_c}}
\newcommand{\critContour}{\ensuremath{\gamma_c}}
\newcommand{\restr}[1]{|_{#1}}

\DeclareMathOperator{\Int}{Int}

\newcommand{\R}{\ensuremath{\mathbb{R}}}
\newcommand{\J}{\ensuremath{J}}

\newcommand{\fourcoeff}[3]{\ensuremath{F^{#3}_{(#1,#2)}}}

\DeclareMathOperator{\length}{len}

\newcommand{\Expec}[1]{\ensuremath{\mathsf{E}\!\left[\vphantom{\big|}#1\vphantom{\big|}\right]}}
\newcommand{\ExpecGOI}[1]{\ensuremath{\mathsf{E}_{\text{GOI(c)}}^n\!\left[\vphantom{\big|}#1\vphantom{\big|}\right]}}
\newcommand{\ExpecGOIp}[3]{\ensuremath{\mathsf{E}_{\text{GOI$\left(#2\right)$}}^{#3}\!\left[\vphantom{\big|}#1\vphantom{\big|}\right]}}
\newcommand{\Var}[1]{\ensuremath{\mathsf{Var}\!\left[\vphantom{\big|}#1\vphantom{\big|}\right]}}
\newcommand{\Cov}[2]{\ensuremath{\mathsf{Cov}\!\left[\vphantom{\big|}#1, #2\vphantom{\big|}\right]}}
\newcommand{\indic}[1]{\ensuremath{\mathsf{1}\left({#1}\right)}}

\DeclareMathOperator{\ind}{ind}

\DeclareMathOperator{\adj}{adj}

\DeclareMathOperator{\real}{Re}
\DeclareMathOperator{\imag}{Im}
\DeclareMathOperator{\vol}{Vol}
\DeclareMathOperator{\codim}{codim}

\usepackage{algorithm}
\usepackage{mathtools}
\usepackage[shortlabels]{enumitem}
\usepackage{algpseudocode}
\usepackage{hyperref}
\usepackage{tikz}
\usetikzlibrary{backgrounds}
\usetikzlibrary{intersections}
\usetikzlibrary{positioning}

\makeatletter \def\BState{\State\hskip-\ALG@thistlm} \makeatother

\usepackage{etoolbox} \makeatletter
\patchcmd{\@maketitle}{\begin{center}}{\begin{flushleft}}{}{}
\patchcmd{\@maketitle}{\begin{tabular}[t]{c}}{\begin{tabular}[t]{@{}l}}{}{}
\patchcmd{\@maketitle}{\end{center}}{\end{flushleft}}{}{}

\usepackage{graphicx}
\usepackage{subcaption}
\usepackage{layouts}

\begin{document}


\title{Singularities of Gaussian Random Maps into the Plane}


\author{Mishal Assif P K}

\maketitle
\abstract{
    We compute the expected value of various quantities related to the biparametric singularities
    of a pair of smooth centered Gaussian random fields on an $n$-dimensional compact manifold,
    such as the lengths of the critical curves and contours of a fixed index and
    the number of cusps. We obtain certain expressions under no particular assumptions other
    than smoothness of the two fields, but more explicit formulae are derived
    under varying levels of additional constraints such as the two random fields being i.i.d,
stationary, isotropic etc.}
\section{Introduction}
\label{sec:intro}

Let $N$ be an $n$-dimensional compact Riemannian manifold ($n \geq 2$). 
Given a smooth function \[ N \ni p \rightarrow h(p) = (f(p),g(p)) \in \R^2, \]
a point $p \in N$ is called a critical point if the derivative $Dh(p):T_xN \rightarrow \R^2$
at $p$ is not surjective and the set of all critical points is called the critical
curve of $h$. The critical point is an example of a singularity of the smooth function
$h$, and the objective of this paper is to study the expected value of various
quantities of interest associated with such singularities when the components
of $h$ are Gaussian random fields (GRFs). The expected number of critical points of a single GRF has been the subject of many papers e.g.
\cite{Che18, Auf13, Auf13c, Adl10, Bar86, Lon60},
having applications in a wide variety of domains.
The singularities of a pair of functions, being a two dimensional analogue of such
one dimensional singularities, naturally warrant study. However, our main motivation
to study these quantities come from biparametric persistent homology of smooth functions.

Persistent homology (PH) is a topological data analysis technique used to extract robust topological
features from data. The key idea in single parameter PH is that if $X$ is a topological space and $f: X \rightarrow \R$
is a nice enough function on $X$, one can encode the change in homologies of the sublevel 
sets $\{f \leq a\}$ as the single parameter $a$ varies along the real line in the form 
of a simple planar diagram called the persistence diagram of $f$. If $X$ is a smooth
manifold and $f$ is a Morse function, it is well known from Morse theory that the
critical points of $f$ are precisely where the homology of its sublevel sets change,
and hence the behavior of critical points of $f$ determine that of the persistence diagram 
of $f$. In biparametric persistence, one has a pair of functions $h = (f,g):X \rightarrow \R^2$
and one tries to track the change in homologies of the sublevel sets $\{f \leq a, g \leq b\}$
as the two parameters $(a,b)$ vary in the plane. When $X$ is a smooth manifold and
the function $h$ is smooth, biparametric persistence can be understood from the perspective
of Whitney theory, analogous to the Morse theoretic perspective of single parameter PH,
and there is a growing amount of literature regarding this \cite{cerri_geometrical_2019,
budney21, bubenik21, baryshnikov21}.

We give a brief description of this Whitney theoretic perspective on biparametric
persistence here, the details of which can be found in \cite{baryshnikov21}. For 
a generic function $h$, the critical curve is a 1-dimensional embedded submanifold of $N$, or a disjoint
finite union of smooth circles. The image of the critical curve under the map $h$
is called the \textit{visible contour}. The visible contour will also be a finite union
of closed curves in $\R^2$, although these curves may intersect each other
and will be smooth only outside a finite number of points called \textit{cusps}. 
The preimage of a cusp point can be characterized as a second order singularity of
$h$, that is, a point of $N$ where the derivatives of $h$ upto order two satisfy certain
conditions. In comparison, critical points are first order singularities of $h$ since
their description only involves conditions on derivatives of $h$ upto order one. Figure
\ref{fig:sub2} shows an example of a visible contour where the critical curve consists
of a single circle. The visible contour is thus a single closed loop in $\R^2$, 
which has one point of self intersection and two cusp points where it loses smoothness.

At the image of the critical points of the component functions $f$ and $g$, the tangents 
to the visible contour are vertical and horizontal respectively. These points thus
split the visible contour into segments with positive or negative slopes. The segments
with negative slope are called \textit{Pareto segments} of the visible contour. The curves
indicated in black in Figure \ref{fig:sub3} are the Pareto segments of the visible
contour shown in Figure \ref{fig:sub2}. At the image of critical points of $f$ and $g$, 
we attach vertical and horizontal rays extending upward and rightward respectively, 
and call them the \textit{extension rays}. The extension rays are the curves marked
in blue in Figure \ref{fig:sub3}. 

The union of the Pareto segments of the visible 
contour and the extension rays is called the Pareto grid of $h$. The
grid formed by the black and blue curves in Figure \ref{fig:sub3} form the Pareto grid
of the visible contour in Figure \ref{fig:sub2}. The Pareto grid has certain additional
points of non-smoothness where a Pareto segment attaches to an extension ray in a 
non-smooth manner and these corner points are called \textit{pseudocusps}. One can see
that there are four non-smooth points on the Pareto grid in Figure \ref{fig:sub3} indicated
in red, and two of these are cusps of the visible contour, while the other two are pseudocusps.
The cusps and pseudocusps split the Pareto grid into multiple smooth pieces. These
smooth pieces of the Pareto grid are the biparametric analogues of critical values 
in single parameter persistent homology. Homology generators are born or killed as
the parameter value $(a, b) \in \R^2$ crosses these smooth pieces. One can define an
index for each of these pieces determining the dimension of the cell attached at a crossing
as well.

\begin{figure}
\centering
\begin{subfigure}{0.3\linewidth}
  \centering
  \includegraphics[width=.8\textwidth]{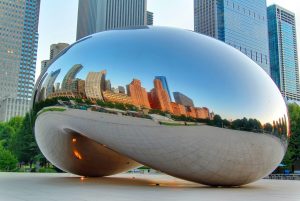}
  \caption{Chicago Millenium Park}
  \label{fig:sub1}
\end{subfigure}
\begin{subfigure}{.28\linewidth}
    \centering
  \includegraphics[width=.6\textwidth]{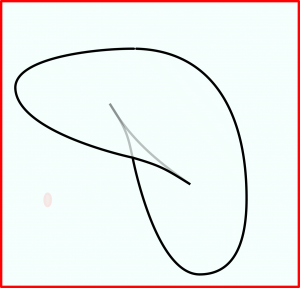}
  \caption{Visible contour}
  \label{fig:sub2}
\end{subfigure}
\begin{subfigure}{0.28\linewidth}
  \centering
\includegraphics[width=0.6\textwidth]{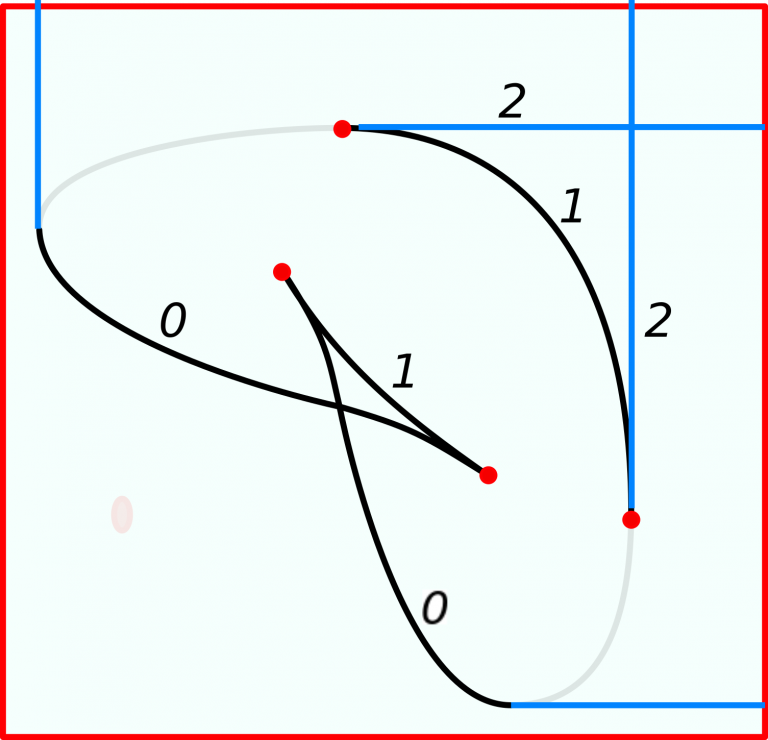}
\caption{Pareto grid} 
\label{fig:sub3}
\end{subfigure}
\caption{\small Figure (a) shows a bean sculpture, the surface of which is a two dimensional sphere.
If $h$ is the projection of the surface onto a plane behind the bean the corresponding
visible contour is shown in Figure (b). The Pareto grid, along with cusps, pseudocusps and
indices of various segments of the grid are shown in Figure (c).}
\label{fig:contour}
\end{figure}

The objective of this article is to study some statistical properties of these biparametric
singularities of a GRF $h$. Given the Whitney theoretic description of biparametric
persistence, understanding the properties of these singularities give
us an idea about the complexity of the biparametric persistence of GRFs into the plane.
We will mainly focus on computing the expected lengths
of the critical curve and visible contour of each index. Unlike critical points, the
cusp points are second order singularities which means their characterization involves
derivatives upto order two. If we were to use the Kac-Rice formula \cite[Theorem 12.1.1]{AdlTay07}
to compute the expected number of cusps, the computations will involve derivatives
of $h$ upto order three making them really cumbersome. Hence, these computations are
left to a later paper. However, the expected number of pseudocusps can be computed
using standard techniques as they are critical points of single variable
GRFs and these computations are done in the article.

We derive expressions for these expectations for general pairs of GRFs, assuming
only that they are smooth and centered (mean zero) and some additional mild technical assumptions.
The expressions yield neater formulae
as more assumptions such as the pair of GRFs being identical and independent,
stationary or isotropic are imposed. However, all the general expressions we find here
are written as expectations of functions of certain Gaussian random vectors and Gaussian
random matrices. The additional assumptions on the GRFs make the distributions of
these random vectors and matrices nicer, such as being independent, rotationally invariant etc.
yielding better closed form solutions.

The paper is structured as follows. The remaining part of section 1 introduces the
notations and assumptions used in the article. In section 2, we derive expressions
for the expected length of the critical curve, visible contour and Pareto segments,
each of a fixed index. This section doesn't assume much about the GRF $h$ beyond 
begin smooth and centered. In section 3, we simplify the expressions obtained in the
previous section under the additional assumption that the component functions $f$
and $g$ are independent and identically distributed GRFs. We also show concrete examples of
these computations in two settings, random bandlimited functions into the plane
and random planar projections of the standard embedded torus in $\R^3$. In section 4,
we obtain neater formulae under the assumptions that the manifold $N$ is the
$n$-sphere $\S^n$ and the component GRFs $f$ and $g$ are isotropic and stationary. 
Finally, in section 5, we derive the expected number of pseudocusps of a fixed index.

\subsection{Notation}
We denote by $N$ an $n$-dimensional compact orientable Riemannian manifold endowed
with a Riemannian metric $G$. The corresponding volume form on $N$ will be denoted
by $d_N V$. We will deal with certain one dimensional compact submanifolds on $N$,
and the volume form endowed by the induced Riemannian metric on them will be 
denoted by $d_N l$. $S^1$ will denote the unit circle endowed with the standard
Riemannian metric. We will also need to look at certain one dimensional compact
submanifolds on the product space $S^1 \times N$ endowed with the product Riemannian
metric and the volume form endowed by the induced Riemannian metric on them will
be denoted by $d_{S^1\times N} l$.

Some of the computations in the article will be done in 
coordinate charts on $N$ and if $v \in \R^n$ is the coordinate representation
of a tangent vector on $N$, $\|v\|_G$ will denote its norm in the Riemannian metric
while $\|v\|$ will denote its usual Euclidean norm. 

$J^r(N, \R^2)$ is the $r$-jet space of $N$ to $\R^2$ and $J_p^r(N, \R^2)$ is the 
$r$-jet space at $p \in N$, which are both Euclidean spaces since the codomain 
$\R^2$ is Euclidean. 
We denote by $S_k$ the corank $k$ submanifold of $J^1(N, \R^2)$
consisting of those jets with rank $2-k$. Given a smooth function $h:N \rightarrow \R^2$,
we denote by $j^rh : N \rightarrow J^r(N, \R^2)$ its $r$-jet function (see \cite{Gol12} for details). 

\[ N \times \Omega \ni (p, \omega) \rightarrow h(p, \omega) = (f(p, \omega),g(p, \omega)) \in \R^2 \]
will denote a Gaussian random field (GRF) on $N$. To make notations less cumbersome,
we will avoid including $\omega$ and refer to $h(p)$ freely as a GRF. The support
of $h$ is defined as
\[
    \supp\left( h \right) = \{ \bar{h} \in C^{\infty}(N, \R^2) \text{ such that } \bP\left(h \in U \right) > 0 \text{ for all neighborhoods $U$ of } \bar{h} \}.
\]
(see \cite{Ste20} for more details).
If $h$ is smooth almost surely, its derivatives are also GRFs and so is $j^r h$.
If $W$ is a submanifold of $J^r(N, \R^2)$, then $j^r h \pitchfork W$ will denote
that the function $j^r h$ intersects $W$ transversally.

\subsection{Assumptions}

There are three standing assumptions throughout the article.

\begin{ass}
    \label{ass:smooth}
    The GRF $h$ is centered, that is, $\Expec{h(p)} = 0$ for all $p \in N$.
\end{ass}
    Our techniques and proofs work identically in the non-centered case
    as well, but the final formulae obtained are not very clean and exact computation
    is not possible when the mean is not zero.
\begin{ass}
    \label{ass:smooth2}
    The GRF $h$ is smooth on $N$ almost surely.
\end{ass}
This is not too strict an assumption, as there are conditions on the GRF
that ensures this happens, such as \cite[Theorem 11.3.4]{AdlTay07}. 

\begin{ass}
    \label{ass:transverse}
    The support of the 2-jet $\normalfont{\supp}\left(j^2 h(p)\right) = J^2_p(N, \R^2)$
    for all $p \in N$, that is, the jointly Gaussian random vector
    \[
        \left(f(p), g(p), \nabla f(p), \nabla g(p), \nabla^2 f(p), \nabla^2 g(p) \right)
    \]
    is non degenerate.
\end{ass}
Smoothness of $h$ isn't the only regularity condition we need for our computations;
we will require the $r$-jet of $h$ satisfy certain non-degeneracy conditions. We will see that
the above assumption on a GRF, along with the following lemma, will be required to ensure this happens almost surely.

\begin{lem}[\cite{Ste20}, Theorem 23]
    \label{l:ste}
    Let $h:N \rightarrow \R^2$ be a smooth GRF and $r \in \N$. Assume that for every $p \in N$ we
    have $ \normalfont{\supp} \left(j^r h(p) \right) = J_p^r(N, \R^2)$. Then for any
    submanifold $W \subset J^r(N, \R^2)$, we have $\mathbb{P}\left(j^r h \pitchfork W \right) = 1.$
\end{lem}

\section{Length computations on general centered GRFs}
\label{sec:exp-length}

\subsection{Expected length of critical curves}

In this section, we derive expressions for the average length of
the critical curve of $h$. 
Recall that a point $p \in N$ is called a critical point of $h$ if the derivative $Dh(p):T_xN \rightarrow \R^2$
at $p$ is not surjective. This is equivalent to saying that the gradient vectors $\nabla f(p), \nabla g(p)$ are not
linearly independent vectors lying in the tangent space $T_pN$. The set of critical
points of $h$ is called the critical set of $h$.
For a generic function $h$\footnote{generic refers
to a set of functions that is open and dense in an appropriate metric on the space of smooth
functions.} the critical set
will be a 1-dimensional embedded submanifold of $N$, or a collection of disjoint
embedded circles in $N$, justifying the term critical \textit{curve}. 

\begin{lem}
    \label{l:morse}
    If $h$ satisfies assumptions \ref{ass:smooth}-\ref{ass:transverse} then the
    critical curve of $h$ is a one dimensional compact submanifold of $N$ on which
    the rank of $Dh$ is 1 almost surely.
\end{lem}
\begin{proof}
    The critical curve $\critCurve = \left(j^1 h \right)^{-1}\left(S_1\cup S_2 \right).$
    This means that the critical curve is closed, as $S_1 \cup S_2$ is a closed subset
    of $J^1(N, \R^2)$.
    As a consequence of Lemma \ref{l:ste}, Assumption \ref{ass:transverse}
    guarantees that 
    \[
        \bP \left( j^1 h \pitchfork W \right) = 1
    \]
    for any submanifold $W \subset J^1(N, \R^2)$. This implies $j^1h \pitchfork S_r$
    for $r = 1, 2$ almost surely. Since $\codim (S_2) = 2n > \dim (N)$, $\left(j^1 h \right)^{-1}S_2$
    is empty. Therefore, the critical curve $\critCurve = \left(j^1 h \right)^{-1}\left(S_1\right).$
    Since $j^1h$ intersects $S_1$ transversally, $\critCurve$ is a submanifold of
    $N$. In addition,
    \[
        \codim\left(S_1\right) = n-1 \implies \codim(\critCurve) = n-1 \implies
        \dim(\critCurve) = 1.
    \]
\end{proof}

We are now in a position where the length of the critical curve makes sense almost
surely, and will derive expressions for the average length in local coordinates
first and then show that the expressions are coordinate invariant.
Let $x = (x_1,x_2,...,x_n)$ be local coordinates on a coordinate neighborhood $U$ of $N$, that is,
\[
    \R^n \ni (x_1,...,x_n) = x \longrightarrow \phi(x) = p \in U
\]
is a diffeomorphism. Let $K \subset U$ be the coordinate image of a compact set
in $\R^n$. To avoid notational clutter, we will refer to the local representation
of $h$ (and $f, g$) as $h$ itself. We can characterize the critical points in
these coordinates as the projection onto $N$ of the zeros of the $\R^n$
valued function
\begin{equation}
    U \times S^1 \ni (x, \theta) \longrightarrow V(x, \theta) =
    \cos(\theta)\nabla f(x) + \sin(\theta)\nabla g(x) \in \R^n.
\end{equation}
We also define the \textit{infinitesimal length vector} $LV(x, \theta)$ as 
\[
    LV(x, \theta) =  \adj\left(\nabla_x V(x,\theta)\right)\nabla_{\theta}V(x,\theta).
\]

The following local result will be our first step.

\begin{prp} If the GRF $h$ satisfies assumptions \ref{ass:smooth}-\ref{ass:transverse}
    and $K$ is a compact set contained in a coordinate chart $U$,
     then the expected value of the length of the critical curve in $K$ is
\label{p:avg len ccur}
\begin{equation}
    \label{eq:avg len ccur loc}
    \Expec{\length(\Sigma_c \cap K)} = \int_K \int_{[0,\pi]}\frac{\Expec{ \|LV(x,\theta)\|_G \bigg| V(x,\theta) = 0}}{\sqrt{(2\pi)^n\det(G(x) \Var{V(x,\theta)} )}} d_N Vd\theta.
\end{equation}
\end{prp}

We will prove \ref{p:avg len ccur} after defining a few more objects and establishing
a sequence of lemmas. Define the $S^1$ indexed family of real valued GRFs $h_{\theta}$ as
\[
    N \ni p \longrightarrow h_{\theta}(p) := \cos(\theta)f(p) + \sin(\theta)g(p) \in \R.
\]

\begin{lem}
    If $h$ is a GRF satisfying assumptions \ref{ass:smooth}-\ref{ass:transverse}, 
    then $h_{\theta}$ is a Morse function outside finitely many pairs $(p, \theta)$, almost surely. The function
    $N \times S^1 \ni (p, \theta) \longrightarrow h_R(\theta, p) := h_{\theta}(p) \in \R$ is also
    a Morse function almost surely.
\end{lem}

\begin{proof}
    Let $W$ be the submanifold of $J^2(N \times S^1, \R)$ defined as
    \[
        W = \big\{j^2 f(p, \theta) \big| \ f:N\times S^1 \longrightarrow \R, \ D_1 f(p, \theta) = 0, \ \rank(D^2_{1,1} f(p,\theta)) = n-1\big\}
    \]
    which satisfies $\codim(W) = n+1$. Assumption \ref{ass:transverse} and Lemma
    \ref{l:ste} tells us that $j^2 h_R \pitchfork W$ almost surely. Therefore,
    $Z = (j^2 h_R)^{-1}(W)$ is a codimension $n+1$ submanifold of $N\times S^1$, i.e.
    a finite set of points. $h_{\theta}$ is not Morse at $(p, \theta)$ iff $(p, \theta) \in Z$,
    which proves the first part of the lemma. 
    
    We use the fact that $h_R$ is a Morse function only if $j^1 h_R \pitchfork S_1 \subset J^1(N \times S^1, \R)$,
    where $S_1$ denotes the corank-1 submanifold of $J^1(N \times S^1, \R)$ \cite[Proposition 6.4]{Gol12}. 
    Assumption \ref{ass:transverse} guarantees that the GRF $h_R : N \times S^1 \longrightarrow \R$ 
    satisfies $\normalfont{\supp}\left(j^1 h_R(p, \theta) \right)= J_{\left(p, \theta \right)}^1(N\times S^1, \R)$.
    Lemma \ref{l:ste} then says that $j^1 h_R \pitchfork S_1 $ almost surely, which means
    $h_R$ is a Morse function almost surely.

\end{proof}

Note that $V(x, \theta)$ is just the derivative of $h_{\theta}$ in local coordinates.
The fact that $h_{\theta}$ is a Morse function ensures that 
\begin{equation}
    \nabla_x V(x, \theta) = \cos(\theta)\nabla^2 f(x) + \sin(\theta)\nabla^2 g(x).
\end{equation}
is non-degenerate at points where $V(x, \theta) = 0$, excluding finitely many points
on the critical curve. Since the length of the critical curve is not affected by removing
finitely many points, we can ignore these points and proceed. This means that $V(x, \theta)$ is a
submersion on $V^{-1}(0)$, which further implies that $V^{-1}(0)$ is a one dimensional
submanifold of $U \times S^1$. If $\pi : N \times S^1 \rightarrow N$ denotes projection
to the first factor, then $\pi (V^{-1}(0)) = \critCurve\cap U$. We denote by $dl$
the volume form on $V^{-1}(0)$ corresponding to the metric induced on it by the
Euclidean metric on $\R^n \times S^1$. Similarly, $d_N l$ denotes the volume form 
on $\critCurve$ corresponding to the metric induced by the Riemannian metric on $N$.

\begin{lem}
    \label{l:det avg len}
    If $h$ is a GRF satisfying assumptions \ref{ass:smooth}-\ref{ass:transverse}, 
    then the following equation holds almost surely.
    \[
        \length(\critCurve \cap K) = \frac{1}{2}\int_{V^{-1}(0) \cap (K\times S^1)} \frac{\| LV(x,\theta)\|_G}{\sqrt{\det(\nabla_x V(x,\theta))^2 + \|LV(x,\theta)\|^2 } } dl.
    \]
\end{lem}
\begin{proof}
The fact that $\nabla_x V(x, \theta)$ is non-degenerate implies that $V^{-1}(0)$
can be locally parametrized by $\theta$, as per the implicit function theorem. 
We will denote this parametrization by $x(\theta)$. The derivative of this function
is then
\[
    \frac{dx}{d\theta} = (\nabla_x V(x,\theta))^{-1}(\nabla_{\theta} V(x,\theta)) = \frac{\adj \left(\nabla_x V(x,\theta) \right) \nabla_{\theta}V(x, \theta)}{\det\left( \nabla_x V(x, \theta) \right) }.
\]
Since $\pi: V^{-1}(0) \rightarrow \critCurve\cap U$ is a double cover of the critical
curve,
\[
    \int_{\critCurve\cap K} d_N l = \frac{1}{2} \int_{V^{-1}(0) \cap (K \times S^1)} \pi^* d_N l.
\]
We know that $(\frac{d x}{d\theta}, 1)$ lies tangent to $V^{-1}(0)$. Then,
\begin{align*}
    & \pi^* d_N l \left(\frac{d x}{d\theta}, 1\right) = d_N l\frac{d x}{d \theta} = \|\frac{d x}{d\theta} \|_G, \quad dl \left(\frac{d x}{d\theta}, 1\right) = \sqrt{ 1 + \|\frac{d x}{d \theta}\|^2 }, \\
    & \implies \pi^* d_N l = \frac{\|\frac{d x}{d\theta} \|_G}{ \sqrt{ 1 + \|\frac{d x}{d \theta}\|^2 } } dl = \frac{\| \adj\left(\nabla_x V(x,\theta)\right)\nabla_{\theta}V(x,\theta)\|_G}{\sqrt{\det(\nabla_x V(x,\theta))^2 + \|\adj\left(\nabla_x V(x,\theta)\right)\nabla_{\theta}V(x,\theta)\|^2 }} dl,
\end{align*}
which proves the result.
\end{proof}

\begin{proof}[Proof of proposition \ref{p:avg len ccur}]
    In lemma \ref{l:det avg len}, we expressed the length of the critical curve
    as a weighted integral over the level set of the real valued GRF $V(x,\theta)$
    on $U \times S^1$. To find the expectation of this weighted integral, we 
    apply the generalized Rice formula (\cite[Theorem 6.10]{Aza09}). Assumption \ref{ass:transverse}
    along with lemma \ref{l:morse} ensure that the conditions required ((i)-(iv)
    in \cite[Theorem 6.8]{Aza09}) to justify this are satisfied. If we denote the total
    derivative of $V$ by
    \[
        \nabla V(x, \theta) = \bigg[ \nabla_x V(x,\theta) \quad \nabla_{\theta}V(x,\theta) \bigg],
    \]
    then 
    \[
        \Expec{\length(\critCurve \cap K)} = \frac{1}{2} \int_K \int_{S^1} \Expec{ \det\left(\nabla V. \nabla V^\top \right)^{\frac{1}{2}} \frac{\| \adj(\nabla_x V) \nabla_{\theta}V \|_G}{\sqrt{\det(\nabla_x V)^2 + \|\adj(\nabla_x V) \nabla_{\theta}V \|^2 }}\bigg| V = 0}p_{V}(0)  dx d\theta,
    \]
    where we have dropped the obvious $(x, \theta)$ dependence to avoid clutter. Observe that
    \begin{align*}
        \det\left(\nabla V.\nabla V^\top\right) &= \det\left(\nabla_xV.\nabla_xV^\top + \nabla_{\theta}V.\nabla_{\theta}V^\top \right) \\
                                                &= \det(\nabla_xV)^2 \det\left(I + \left(\nabla_x V^{-1}\nabla_{\theta}V\right) \left(\nabla_x V^{-1} \nabla_{\theta}V \right)^\top\right) \\
                                                &= \det(\nabla_xV)^2 \left(1 + \|\nabla_xV^{-1} \nabla_{\theta}V\|^2\right) = \det(\nabla_xV)^2 + \|\adj(\nabla_xV)\nabla_{\theta}V \|^2.
    \end{align*}
    In addition, $p_{V(x,\theta)}(0)$ is the density of $V(x,\theta)$ evaluated at $0$. Since
$V(x,\theta)$ is an $n$-dimensional mean zero Gaussian random vector, 
\[
    p_{V(x,\theta)}(0) = (2\pi)^{\frac{-n}{2}}(\det(\Var{V(x,\theta)})^{\frac{-1}{2}}.
\]
In total, we get
    \[
        \Expec{\length(\critCurve \cap K)} = \frac{1}{2 (2\pi)^{\frac{n}{2}} } \int_K \int_{S^1} \frac{\Expec{ \| \adj(\nabla_x V(x,\theta)) \nabla_{\theta}V(x,\theta) \|_G \bigg| V(x,\theta) = 0}}{\sqrt{\det(\Var{V(x,\theta)})}}  dx d\theta.
    \]
The Riemannian volume form on $N$ is related to the Euclidean volume form $dx$ as
\[
    d_N V = \sqrt{\det(G(x))} dx.
\]
In addition, observe that $\nabla_x V(x, \theta + \pi) = -\nabla_x V(x, \theta),
\nabla_{\theta} V(x, \theta + \pi) = -\nabla_{\theta} V(x, \theta) $. Therefore, we can say that
    \[
        \Expec{\length(\critCurve \cap K)} = \int_K \int_{[0,\pi]} \frac{\Expec{ \| \adj(\nabla_x V(x,\theta)) \nabla_{\theta}V(x,\theta) \|_G \bigg| V(x,\theta) = 0}}{\sqrt{(2\pi)^n\det(G(x)\Var{V(x,\theta)})}}  d_N V d\theta.
    \]
\end{proof}

We now show that the integrand in \eqref{eq:avg len ccur loc} is coordinate invariant.
Let 
\[
    \R^n \ni (y_1,...,y_n) = y \longrightarrow \psi(y) \in U
\]
be another coordinate chart on $U$. Let $J$ be the Jacobin of the coordinate change
map from $(y_1,...,y_n) \rightarrow (x_1,...,x_n)$. Then
\[
    V(x, \theta) = J^\top V(y, \theta), \quad \nabla_{\theta} V(x, \theta) = J^\top V(y, \theta), \quad \nabla_x V(x, \theta) = J^\top \nabla_y V(y, \theta) J.
\]
The Riemannian metric tensor transforms as $G(x) = J^\top G(y) J$. This means
\begin{align*}
    \adj(\nabla_x V(x, \theta)) \nabla_{\theta} V(x, \theta) &= \det(J)^2 J^{-1}\adj(\nabla_y V(y, \theta)J^{-\top} J^\top \nabla_{\theta}V(y, \theta) \\
                                                             &= \det(J)^2 J^{-1} \nabla_y V(y, \theta) \nabla_{\theta} V(y,\theta), \\
     \sqrt{\det(G(x)\Var{V(x,\theta)})} &= \det(J)^2 \sqrt{\det(G(y)\Var{V(y,\theta)}}.
\end{align*}
We can write
\begin{align*}
    \| \adj(\nabla_x V(x, \theta)) &\nabla_{\theta} V(x, \theta) \|_G =  \nabla_{\theta}V(x,\theta)^\top \adj(\nabla_x V(x,\theta))^\top G(x)\adj(\nabla_x V(x, \theta)) \nabla_{\theta} V(x, \theta) \\
                                                                     &=  (\det(J))^2\nabla_{\theta}V(y,\theta)^\top \adj(\nabla_y V(y,\theta))^\top J^{-\top}G(x)J^{-1}\adj(\nabla_y V(y, \theta)) \nabla_{\theta} V(y, \theta) \\
                                                                     &=  (\det(J))^2\nabla_{\theta}V(y,\theta)^\top \adj(\nabla_y V(y,\theta))^\top G(y) \adj(\nabla_y V(y, \theta)) \nabla_{\theta} V(y, \theta) \\
                                                                     &= (\det(J))^2\| \adj(\nabla_y V(y, \theta)) \nabla_{\theta} V(y, \theta) \|_G 
\end{align*}
which immediately gives coordinate invariance of the integrand. We can now give each
term in the integrand the following coordinate invariant characterization:
\begin{align*}
    V(x, \theta) &\longrightarrow \nabla h_{\theta}(p), \nabla_x V(x, \theta) \longrightarrow \nabla^2 h_{\theta}(p), \nabla_{\theta} V(x, \theta) \longrightarrow  \nabla h_{\theta+\frac{\pi}{2}}(p).
\end{align*}
If we define 
\[
    \left[ \Var{\nabla h_{\theta}(p)} \right]_{i,j} := \Expec{(\nabla h_{\theta} (p). v_i)(\nabla h_{\theta} (p). v_j) }
\]
where $\{v_i\}_{i=1}^n$ is an orthonormal basis of $T_pN$, then
\[
    \det(\Var{\nabla h_{\theta}(p)} ) = \det(G(x) \Var{V(x, \theta)} ).
\]
We can now extend proposition \ref{p:avg len ccur} globally to get the main result of
this section.
\begin{thm} If the GRF $h$ satisfies assumptions \ref{ass:smooth}-\ref{ass:transverse},
    then the expected value of the length of its critical curve is
\label{th:avg len ccur}
\begin{equation}
    \label{eq:avg len ccur}
    \Expec{\length(\Sigma_c)} = \int_N \int_{[0,\pi]}\frac{\Expec{ \|\adj\left(\nabla^2 h_{\theta}(p)\right) \nabla h_{\theta+\frac{\pi}{2}}(p)\|_G \bigg| \nabla h_{\theta}(p) = 0}}{\sqrt{(2\pi)^n\det\left(\Var{\nabla h_{\theta}(p)} \right)}} dVd\theta.
\end{equation}
\end{thm}
\begin{proof}
    Cover the manifold $N$ by a finite number of compact coordinate disks $K_i, \ i = 1,...,k$. We
    already know
    \begin{align*}
        \Expec{\length(\Sigma_c)\cap K_i} = \int_{K_i} \int_{[0,\pi]}\frac{\Expec{ \|\adj\left(\nabla^2 h_{\theta}(p)\right) \nabla h_{\theta+\frac{\pi}{2}}(p)\|_G \bigg| \nabla h_{\theta}(p) = 0}}{\sqrt{(2\pi)^n\det\left(\Var{\nabla h_{\theta}(p)} \right)}} dVd\theta.
    \end{align*}
    By the inclusion exclusion principle,
    \[
        \Expec{\length{\critCurve}} = \sum\limits_{l=1}^k \sum_{i_1 < i_2 < ... < i_l} (-1)^{l+1} \Expec{\length{\critCurve \cap K_{i_1} \cap K_{i_2} ... \cap K_{i_l}}}.
    \]
    But the integral of any function $f:N \rightarrow \R$ can be written as
    \[
        \int_N f d_N V = \sum\limits_{l=1}^k \sum_{i_1 < i_2 < ... < i_l} (-1)^{l+1} \int_{\cap_{j=1}^{l}K_{i_j}} f d_N V,
    \]
    from which the result follows directly.
\end{proof}

\begin{rem}
    Notice that $(\nabla^2 h_{\theta}, \nabla h_{\theta+\frac{\pi}{2}}, \nabla h_{\theta})$ are jointly Gaussian random vectors
and so the conditional expectation in \eqref{eq:avg len ccur} is just the expectation
of a function of a Gaussian random vector.
\end{rem}

\subsection{Expected length of the visible contour}

A point $v \in \R^2$
is called a critical value of $h$ if the preimage $h^{-1}\{v\}$ contains a critical point,
that is, $\critContour = h(\critCurve)$.
The subset of $\R^2$ consisting of all critical values is called the visible contour $\gamma_c$
of $h$. In this section, we will compute the average length of the visible contour of $h$ in this section.

The computation of the expected length of the visible contour follows the exact
same procedure as the one we saw in the previous section.  We have
the following analogues of proposition \ref{p:avg len ccur}, lemma \ref{l:det avg len}.

\begin{prp} If the GRF $h$ satisfies assumptions \ref{ass:smooth}-\ref{ass:transverse}
    and $K$ is a compact set contained in a coordinate chart $U$,
     then the expected value of the length of its critical curve in $K$ is
\label{p:avg len ccont}
\begin{equation}
    \label{eq:avg len ccont loc}
    \Expec{\length(h(\Sigma_c \cap K))} = \int_K \int_{[0,\pi]}\frac{\Expec{ \left|\nabla_{\theta} V(x,\theta)^\top LV(x,\theta)\right| \bigg| V(x,\theta) = 0}}{\sqrt{(2\pi)^n\det(G(x) \Var{V(x,\theta)} )}} d_N Vd\theta.
\end{equation}
\end{prp}

\begin{lem}
    \label{l:det avg len 2}
    If $h$ is a GRF satisfying assumptions \ref{ass:smooth}-\ref{ass:transverse}, then the following
    equation holds almost surely.
    \[
        \length(\critCurve \cap K) = \frac{1}{2}\int_{V^{-1}(0) \cap (K\times S^1)} \frac{\left|\nabla_{\theta}V(x,\theta)^\top LV(x,\theta)\right|}{\sqrt{\det(\nabla_x V(x,\theta))^2 + \|LV(x,\theta)\|^2 } } dl.
    \]
\end{lem}
\begin{proof}

Here we have the sequence of maps
\[
    V^{-1}(0) \overset{\pi}{\longrightarrow} \Sigma_c \cap U \overset{h}{\longrightarrow} \gamma_c.
\]
If we denote by $d_{\R^2} l$ the volume form on $\gamma_c$ 
outside a finite set of cusps)
induced from the Euclidean metric on $R^2$, we can say that
\[
    \int_{h(\Sigma_c \cap K)} d_{\R^2} l = \frac{1}{2} \int_{V^{-1}(0)\cap(K\times S^1)} (h\circ \pi)^{*} d_{\R^2} l.
\]

We know that $(\frac{d x}{d\theta}, 1)$ lies tangent to $V^{-1}(0)$. Then,
\begin{align*}
    & (h\circ\pi)^* d_{\R^2} l \left(\frac{d x}{d\theta}, 1\right) = d_{\R^2} l \left(Dh\frac{d x}{d \theta}\right) = \|Dh\frac{d x}{d\theta} \|, \quad dl \left(\frac{d x}{d\theta}, 1\right) = \sqrt{ 1 + \|\frac{d x}{d \theta}\|^2 }.
\end{align*}

Since the Euclidean norm on $\R^2$ is 
invariant under rotation by an angle $\theta$, we can say
\begin{equation}
   \|Dh\frac{d x}{d \theta}\| =  \left\| \begin{bmatrix}
        \left(\cos(\theta)\nabla f(x) + \sin(\theta)\nabla g(x)\right)^{\top} \\
        \left(-\sin(\theta)\nabla f(x) + \cos(\theta)\nabla g(x)\right)^{\top} 
\end{bmatrix}   \frac{dx}{d\theta}\right\|.
\end{equation}
However, since $\cos(\theta)\nabla f(x) + \sin(\theta)\nabla g(x) = 0$, we can 
further rewrite this term as
\begin{equation}
    \left|\left\langle -\sin(\theta)\nabla f(x) + \cos(\theta)\nabla g(x), \frac{dx}{d\theta} \right\rangle \right|
    = \left|\nabla_{\theta}V(x,\theta)^{\top}(\nabla_xV(x,\theta))^{-1}\nabla_{\theta}V(x,\theta)\right|.
\end{equation}
Therefore,
\[
    (h\circ\pi)^* d_{\R^2} l = \frac{\|Dh\frac{d x}{d\theta} \|}{ \sqrt{ 1 + \|\frac{d x}{d \theta}\|^2 } } = \frac{|\nabla_{\theta}V(x,\theta)^\top \adj(\nabla_x V(x,\theta)\nabla_{\theta}V(x,\theta)|}{\sqrt{\det(\nabla_x V(x,\theta))^2 + \|\adj(\nabla_x V(x,\theta)\nabla_{\theta}V(x,\theta)\|^2 }} dl.
\]
from which the result follows through the exact same steps as in the proof of lemma \ref{l:det avg len}.
\end{proof}

The rest of the computations and justifications in the proof of \ref{p:avg len ccont} also follow the exact
same pattern as in the proof of \ref{p:avg len ccur}, and won't be repeated. The main result of this
section also follows from \ref{p:avg len ccont} as

\begin{thm} If the GRF $h$ satisfies assumptions \ref{ass:smooth}-\ref{ass:transverse},
    then the expected value of the length of its visible contour is
\label{th:avg len ccur 2}
\begin{equation}
    \label{eq:avg len ccur 2}
    \Expec{\length(\gamma_c)} = \int_N \int_{[0,\pi]}\frac{\Expec{ \left| \nabla h_{\theta+\frac{\pi}{2}}^\top \adj\left(\nabla^2 h_{\theta}(p)\right) \nabla h_{\theta+\frac{\pi}{2}}(p) \right| \bigg| \nabla h_{\theta}(p) = 0}}{\sqrt{(2\pi)^n\det\left(\Var{\nabla h_{\theta}(p)} \right)}} dVd\theta.
\end{equation}
\end{thm}

\begin{rem}
    The length of the visible contour does not depend on the choice of metric
    on $N$. Indeed, the formula given in equation \eqref{eq:avg len ccur 2}
    is invariant under a change of metric, since
    $\frac{dV}{\sqrt{\det G}}$ does not depend on the choice of metric.
\end{rem}

\begin{rem}
The Pareto segments of the visible contour are the parts of the contour where
    its slope is negative. The segments of the contour play a special role in biparametric persistent 
    homology and the length of only these parts can also be easily computed. The 
    only observation needed to do this is that $h(x)$ lies on the Pareto segment
    only if $\theta$ lies in $[0, \frac{\pi}{2}]$. So one simply needs to replace the
    domain of evaluation of the inner $\theta$ integral in \eqref{eq:avg len ccur 2}
    with $[0, \frac{\pi}{2}]$.
\end{rem}

\subsection{Expected length of segments of fixed index} 
    
The index of a critical point (and the corresponding critical value) of a single function $f$ refers
to the index of the Hessian of $f$ at the critical point. The significance of the 
index is that if $f(p)$ is a critical value of index $k$, then the sublevel set $\{f \leq f(p)+\epsilon \}$
is obtained by attaching a $k$-cell to $\{f \leq f(p) - \epsilon \}$ as long as $(f(p)-\epsilon, f(p)+\epsilon)$
contains no other critical values. In the single persistent homology setting, this
translates to the fact that an index $k$ critical value can either lead to a birth
in $k$-th homology or a death in $k+1$-th homology. 

The Pareto segments of the visible contour plays a similar role in bi-biparametric
persistence. If $p$ is a critical point of $h$  and the corresponding value $v= (f(p),g(p))$ is a point
on a Pareto segment of the visible contour of index
$k$ (to be defined soon), then the sublevel set $\{f \leq f(p) +\epsilon, g \leq g(p) + \epsilon \}$
is obtained by attaching a $k$-cell to $\{f \leq f(p) - \epsilon, g \leq g(p) - \epsilon \}$.
If $\nabla g(p) \neq 0$, then $p$ will be a critical point of the function $f$
restricted to the submanifold $\{g = g(p)\}$, and the appropriate definition of
the index of $p$ is just the index of the Hessian of $f$ restricted to $\{g=g(p)\}$
at $p$. We have the following characterization of the biparametric index.

\begin{prp}
Suppose $\nabla g(p) \neq 0$, and $p$ is a critical point of the function $f$
restricted to the submanifold $\{g = g(p)\}$. Then the index of this critical point
is given by
\[
    \ind\left(x, \theta \right) = \ind{\left( \nabla_x V(x,\theta)\right)} - \indic{\nabla_{\theta}V(x,\theta)^{\top}\left(\nabla_x V(x,\theta)\right)^{-1}\nabla_{\theta} V(x, \theta) < 0 }. \\
\]
\end{prp}
\begin{proof}
As mentioned before, the biparametric index is just the usual index of $f$ restricted
to the level set $\{g = g(p)\}$, which can be computed using the method of Lagrange multipliers as
follows. Define the Lagrangian
\[
    L(x, \lambda) := f(x) + \lambda(g(x) - g(p)).
\]
If $p$ is a critical point and $\nabla g(p) \neq  0$, then there exists some multiplier
$\lambda^*$ such that $\nabla f(p) + \lambda^* \nabla g(p) = 0$. This means that
$(p, \lambda^*)$ is a critical point of $L$. The restricted index of $p$ is then
just \[\ind{\left(\nabla^2 L(p,\lambda^*)\right)}-1\]
where
\[
    \nabla^2 L(p,\lambda^*) = \begin{bmatrix}
        \nabla^2 f(p) & \nabla g(p) \\
        \nabla g(p)^{\top} & 0 
    \end{bmatrix}.
\]
In order to symmetrize our computations, we can use the fact that the index
can also be computed using the index of the Lagrangian
\[
    L(x, \lambda) = \left(\cos(\theta)f(x) + \sin(\theta)g(x)\right)
    + \lambda\left( -\sin(\theta)f(x) + \cos(\theta)g(x) \right)
\]
so that
\begin{align*}
    \nabla^2 L(p,\lambda^*) &= \begin{bmatrix}
        \cos(\theta)\nabla^2 f(p)+\sin(\theta)\nabla^2 g(p) &  -\sin(\theta)\nabla f(p) + \cos(\theta)\nabla g(p) \\
        -\sin(\theta)\nabla f(p)^{\top} + \cos(\theta)\nabla g(p)^{\top} & 0 
    \end{bmatrix} \\
                            &= \begin{bmatrix}
                                \nabla_x V(x, \theta) &  \nabla_{\theta} V(x, \theta) \\
                                \nabla_{\theta}V(x,\theta)^{\top} & 0 
    \end{bmatrix}. 
\end{align*}
To see why this is true, imagine what happens to the sublevel set when crossing
the visible contour at a direction normal to it as opposed to the horizontal
direction; the dimension of the cell attached must be the same in both cases.

The above matrix is conjugate to
\begin{align*}
    \begin{bmatrix}
        \nabla_x V(x, \theta) & 0 \\
        0 & -\nabla_{\theta}V(x,\theta)^{\top}\left(\nabla_x V(x,\theta)\right)^{-1}\nabla_{\theta} V(x, \theta) 
    \end{bmatrix}
\end{align*}
and since index is invariant under change of basis, the index of $p$ is just
\begin{align*}
    &\ind{\left( \nabla_x V(x,\theta)\right)} + \indic{\nabla_{\theta}V(x,\theta)^{\top}\left(\nabla_x V(x,\theta)\right)^{-1}\nabla_{\theta} V(x, \theta) > 0 } - 1 \\
    &=\ind{\left( \nabla_x V(x,\theta)\right)} - \indic{\nabla_{\theta}V(x,\theta)^{\top}\left(\nabla_x V(x,\theta)\right)^{-1}\nabla_{\theta} V(x, \theta) < 0 }.
\end{align*}
\end{proof}
The index at a point again is just a function of $(\nabla_{\theta} V, \nabla_x V)$,
which we denote as $\ind(x,\theta)$. We denote the segments of the critical
curve and contour of index $k$ by $\Sigma_c^k$ and $\gamma_c^k$ respectively.
We then have the obvious analogues of lemma \ref{l:det avg len}, \ref{l:det avg len 2},
which we state without proving.

\begin{lem}
    \label{l:det avg len ind}
    If $h$ is a GRF satisfying assumptions \ref{ass:smooth}-\ref{ass:transverse}, then the following
    equations hold almost surely.
    \begin{align*}
        \length(\critCurve^k \cap K) &= \frac{1}{2}\int_{V^{-1}(0) \cap (K\times S^1)} \frac{\|LV(x,\theta)\|_G \indic{\ind(x, \theta) = k} }{\sqrt{\det(\nabla_x V(x,\theta))^2 + \|LV(x,\theta)\|^2 } } dl, \\
        \length(h(\Sigma_c^k \cap K)) &= \frac{1}{2}\int_{V^{-1}(0) \cap (K\times S^1)} \frac{\left|\nabla_{\theta}V(x,\theta)^\top LV(x,\theta)\right| \indic{\ind(x, \theta) = k} }{\sqrt{\det(\nabla_x V(x,\theta))^2 + \|LV(x,\theta)\|^2 } } dl.
    \end{align*}
\end{lem}

\begin{prp} If the GRF $h$ satisfies assumptions \ref{ass:smooth}-\ref{ass:transverse}
    and $K$ is a compact set contained in a coordinate chart $U$,
     then the expected value of the length of the critical curve of index $k$ in $K$ is
\label{p:avg len ccur ind}
\begin{equation}
    \label{eq:avg len ccur loc ind}
    \Expec{\length(\Sigma_c^k \cap K)} = \int_K \int_{[0,\pi]}\frac{\Expec{ \|LV(x,\theta)\|_G \indic{\ind(x, \theta) = k} \bigg| V(x,\theta) = 0}}{\sqrt{(2\pi)^n\det(G(x) \Var{V(x,\theta)} )}} d_N Vd\theta,
\end{equation}
and that of the visible contour of index $k$ in $K$ is
\begin{equation}
\begin{aligned}
    \label{eq:avg len ccont loc ind}
    &\Expec{\length(h(\Sigma_c^k \cap K))} = \\ &\int_K \int_{[0,\pi]}\frac{\Expec{ \left|\nabla_{\theta}V(x,\theta)^\top LV(x,\theta)\right| \indic{\ind(x, \theta) = k} \bigg| V(x,\theta) = 0}}{\sqrt{(2\pi)^n\det(G(x) \Var{V(x,\theta)} )}} d_N Vd\theta.
\end{aligned}
\end{equation}
\end{prp}
\begin{proof}
    The sets
    \begin{align*}
        O_k^{>} &:= \{ (M, v) \in \Sym{n} \times \R^n \big| M \text{ invertible, }\ind(M) = k, v^\top M^{-1}v > 0\}, \\
        O_k^{<} &:= \{ (M, v) \in \Sym{n} \times \R^n \big| M \text{ invertible, }\ind(M) = k, v^\top M^{-1}v < 0\}
    \end{align*}
    are open in $\Sym{n} \times \R^n$. Observe that
    \begin{align*}
        \indic{\ind(x,\theta) = k} = \indic{(\nabla_x V(x, \theta), \nabla_{\theta} V(x, \theta)) \in O_k^{>} \cup O_{k+1}^{<}}.
    \end{align*}
    The indicator function of any open set can be approximated pointwise by
    a sequence of bounded continuous functions, which means there exists continuous
    bounded functions
    \[
        \Sym{n} \times \R^n \ni (M, v) \longrightarrow \mathsf{1}_{\epsilon}^k(M, v) \in [0,1]
    \]
    such that 
    \[
        \mathsf{1}^k_{\epsilon}(M, v) \uparrow \indic{(M, v) \in O_k^{>}\cup O_{k+1}^{<}} \text{ everywhere, as } \epsilon \downarrow 0.
    \]
    If we now define
    \begin{align*}
        \mathsf{1}_{\epsilon}(\ind(x,\theta) = k) = \mathsf{1}^k_{\epsilon}(\nabla_x V(x, \theta), \nabla_{\theta} V(x, \theta)),
    \end{align*}
    and
    \begin{align*}
        \length_{\epsilon}(\critCurve^k \cap K) &= \frac{1}{2}\int_{V^{-1}(0) \cap (K\times S^1)} \frac{\|\adj(\nabla_x V(x,\theta)\nabla_{\theta}V(x,\theta)\|_G \mathsf{1}_{\epsilon}({\ind(x, \theta) = k}) }{\sqrt{\det(\nabla_x V(x,\theta))^2 + \|\adj(\nabla_x V(x,\theta)\nabla_{\theta}V(x,\theta)\|^2 } } dl, \\
    \end{align*}
    then by the monotone convergence theorem,
    \begin{align*}
        \length_{\epsilon}(\critCurve^k \cap K) \uparrow \length(\critCurve^k \cap K) \text{ almost surely as $\epsilon \to 0$. }
    \end{align*}
    We need this continuous bounded approximation for the application of \cite[Theorem 6.10]{Aza09}
    in the proof of proposition \ref{p:avg len ccur}. The same steps as in the proof
    of proposition \ref{p:avg len ccur} now give
    \begin{align*}
        \Expec{\length_{\epsilon}(\Sigma_c^k \cap K)} = \int_K \int_{[0,\pi]}\frac{\Expec{ \|\adj(\nabla_x V(x,\theta))(\nabla_{\theta} V(x,\theta)\|_G \mathsf{1}_{\epsilon}(\ind(x, \theta) = k) \bigg| V(x,\theta) = 0}}{\sqrt{(2\pi)^n\det(G(x) \Var{V(x,\theta)} )}} d_N Vd\theta.
    \end{align*}
    Applying the monotone convergence theorem on both sides and to the conditional 
    expectation in the integrand, we get the required result. The proof for the
    visible contour follows exactly the same way.
\end{proof}

We can now globalize the above result using the exact same arguments as in the
proof of \ref{th:avg len ccur} to get,

\begin{thm} If the GRF $h$ satisfies assumptions \ref{ass:smooth}-\ref{ass:transverse},
    then the expected value of the length of its critical curve and contour
    of index $k$ are
\label{th:avg len ccur ind 1}
\begin{equation}
    \label{eq:avg len ccur ind 2}
    \Expec{\length(\Sigma^k_c)} = \int_N \int_{[0,\pi]}\frac{\Expec{ \|\adj\left(\nabla^2 h_{\theta}(p)\right) \nabla h_{\theta+\frac{\pi}{2}}(p) \|_G\indic{\ind(p,\theta) = k} \bigg| \nabla h_{\theta}(p) = 0}}{\sqrt{(2\pi)^n\det\left(\Var{\nabla h_{\theta}(p)} \right)}} dVd\theta,
\end{equation}
and
\begin{equation}
    \label{eq:avg len ccont ind 2}
    \Expec{\length(\gamma^k_c)} = \int_N \int_{[0,\pi]}\frac{\Expec{ \left| \nabla h_{\theta+\frac{\pi}{2}}^\top \adj\left(\nabla^2 h_{\theta}(p)\right) \nabla h_{\theta+\frac{\pi}{2}}(p) \right|\indic{\ind(p,\theta) = k} \bigg| \nabla h_{\theta}(p) = 0}}{\sqrt{(2\pi)^n\det\left(\Var{\nabla h_{\theta}(p)} \right)}} dVd\theta.
\end{equation}
\end{thm}

\section{Length computations on independent and identical GRFs}
\label{sec:iid}

All the expressions derived for the expectation of average length depend on the joint
distribution of $(\nabla_{\theta} V(x,\theta), \nabla_x V(x,\theta))$ conditioned on $V(x, \theta) = 0$. 
In this section, we see that if we assume $f$ and $g$ are identical and independent random processes, the conditional
distribution does not depend on $\theta$ allowing us to get rid of the $\theta$ integral
in our formulae. We will also see two concrete examples of computations assuming
i.i.d pairs in this section.

\begin{thm} If the GRF $h$ satisfies assumptions \ref{ass:smooth}-\ref{ass:transverse}
    and its components $f$ and $g$ are independent and identically
    distributed GRFs,
    then the expected value of the length of its critical curve and contour
    of index $k$ are
\label{th:avg len ccur ind 2}
\begin{equation}
    \label{eq:nice avg len}
    \Expec{\length{\Sigma_c^k}} = \int_N \frac{\pi \Expec{ \|\adj\left(\nabla^2f(p)-\Expec{\nabla^2f(p) \big| \nabla f(p)}\right)\nabla f(p)\|_G\indic{\ind(p)=k} }}{\sqrt{(2\pi)^n \Var{\nabla f(p)} )}} dV,
\end{equation}
and
\begin{equation}
    \label{eq:nice avg len 2}
    \begin{aligned}
    \Expec{\length{\gamma_c^k}} = \int_N \frac{\pi \Expec{ \left|\nabla f(p)^{\top}\adj\left(\nabla^2f(p)-\Expec{\nabla^2f(p) \big| \nabla f(p)}\right)\nabla f(p)\right|\indic{\ind(p)=k} }}{\sqrt{(2\pi)^n \Var{\nabla f(p)} )}} dV,
\end{aligned}
\end{equation}
where
\begin{align*}
    \ind(p) = \ind{\left(\nabla^2 f(p)\right)} - \indic{\nabla f(p)^{\top} \left(\nabla^2 f(p)\right)^{-1}\nabla f(p) < 0 }.
\end{align*}
\end{thm}
\begin{proof}
Observe that if the pair $(f,g)$ is i.i.d,
\[
    \Var{V(x,\theta)} = \Var{\cos(\theta)\nabla f(x) + \sin(\theta)\nabla g(x)} = \Var{\nabla f(x)}.
\]
In addition, see that
\begin{align*}
    \nabla_{\theta} V(x,\theta) &= -\sin(\theta)\nabla f(x) + \cos(\theta)\nabla g(x), \ \nabla_x V(x,\theta) = \cos(\theta) \nabla^2 f(x) + \sin(\theta)\nabla^2 g(x),
\end{align*}
which means
\begin{align*}
    \Var{\nabla_{\theta} V(x,\theta)} = \Var{\nabla f(x)}, \ \Var{\nabla_x V(x,\theta)} = \Var{\nabla^2 f(x)}, \
\end{align*}
and
\[\Cov{\nabla_{\theta}V(x,\theta)}{V(x,\theta)} = 0, \  
\Cov{\nabla_{\theta}V(x,\theta)}{\nabla_xV(x,\theta)} = 0. \]
Putting all this together, we can say that $(\nabla_{\theta}V(x,\theta), \nabla_xV(x,\theta))$
are conditionally independent given $V(x,\theta) = 0$, and that
\begin{align*}
    \Var{\nabla_{\theta}V(x,\theta) \big| V(x,\theta) = 0} &= \Var{\nabla f(x)}, \\
    \Var{\nabla_{x}V(x,\theta) \big| V(x,\theta) = 0} &= \Var{\nabla^2 f(x) | \nabla f(x) = 0 }.
\end{align*}
This is just the joint distribution of $\left(\nabla f(x), \nabla^2 f(x) - \Expec{\nabla^2f(x) \big| \nabla f(x)}\right)$.

So the expected length of the critical curve can be rewritten as
\begin{equation*}
    \Expec{\length(\Sigma_c)} = \int_N \frac{\pi\Expec{ \|\adj\left(\nabla^2f(p)-\Expec{\nabla^2f(p) \big| \nabla f(p)}\right)\nabla f(p)\|_G }}{\sqrt{(2\pi)^n \Var{\nabla f(p)} )}} dV,
\end{equation*}
 and that of the contour as
\begin{equation*}
    \Expec{\length(\gamma_c)} = \int_N \frac{\pi \Expec{ \left|\nabla f(p)^{\top}\adj\left(\nabla^2f(p)-\Expec{\nabla^2f(p) \big| \nabla f(p)}\right)\nabla f(p)\right| }}{\sqrt{(2\pi)^n \Var{\nabla f(p)} )}} dV.
\end{equation*}
The index of a critical point can also be simplified as
\[
    \ind(x) = \ind{\left(\nabla^2 f(x)\right)} - \indic{\nabla f(x)^{\top} \left(\nabla^2 f(x)\right)^{-1}\nabla f(x) < 0 }
\]
giving the length of index $k$ segments as
\begin{align*}
    \Expec{\length{\Sigma_c^k}} = \int_N \frac{\pi \Expec{ \|\adj\left(\nabla^2f(p)-\Expec{\nabla^2f(p) \big| \nabla f(p)}\right)\nabla f(p)\|_G\indic{\ind(p)=k} }}{\sqrt{(2\pi)^n  \Var{\nabla f(p)} )}} dV
\end{align*}
and
\begin{align*}
    \Expec{\length{\gamma_c^k}} = \int_N \frac{\pi \Expec{ \left|\nabla f(p)^{\top}\adj\left(\nabla^2f(p)-\Expec{\nabla^2f(p) \big| \nabla f(p)}\right)\nabla f(p)\right|\indic{\ind(p)=k} }}{\sqrt{(2\pi)^n  \Var{\nabla f(p)} )}} dV.
\end{align*}
\end{proof}

\subsection{Examples}
We now see some concrete computations of these expectations. We don't show the index 
computations here either, as the i.i.d assumption is still not enough to get adequate
structure on the Hessian of $f$ for index computations. We will however do this in the next
section while assuming isotropy.
\begin{example}[Bandlimited functions on the flat torus]
    We choose $N$ to be the 2-Torus identified as the quotient space of the unit
    square equipped with the standard flat metric on the unit square. The computations
    here will be done in the usual Euclidean coordinates of the unit square. The Riemannian
    metric tensor $G(x)$ in this coordinate system is just the 2x2 identity matrix.
    See Figure \ref{fig:subfigures} for a few examples of critical curves of such bandlimited
    functions.

    Consider bandlimited functions with random Fourier coefficients
    \begin{align*}
        f(x, y) &= \sum_{(m,n) \in [-K, K]^2} \fourcoeff{m}{n}{1} e^{2\pi i mx}e^{2\pi i ny}, \\
        g(x, y) &= \sum_{(m,n) \in [-K, K]^2} \fourcoeff{m}{n}{2} e^{2\pi i mx}e^{2\pi i ny}.
    \end{align*}
    with the assumptions
    \begin{enumerate}
        \item $\fourcoeff{m}{n}{j} = \overline{\fourcoeff{-m}{-n}{j}}$, so that $(f, g)$ is real.
        \item $\real{\fourcoeff{m}{n}{j}}$ and $\imag{\fourcoeff{m}{n}{j}}$ are i.i.d with $\Expec{|\fourcoeff{m}{n}{j}|^2} = 1,$
        \item $\{ \fourcoeff{m}{n}{j} \}_{(m \geq 0, j = 1,2)}$ are pairwise independent.
    \end{enumerate}
    The above conditions ensure that $f$ and $g$ are identical and independent processes. 
    We now compute
    \begin{align*}
        \frac{\partial f(x, y)}{\partial x} &= \sum_{(m,n) \in [-K, K]^2} 2\pi i m\fourcoeff{m}{n}{1} e^{2\pi i mx}e^{2\pi i ny}, \\
        \frac{\partial f(x, y)}{\partial y} &= \sum_{(m,n) \in [-K, K]^2} 2\pi i n\fourcoeff{m}{n}{1} e^{2\pi i mx}e^{2\pi i ny}, \\
        \frac{\partial^2 f(x, y)}{\partial^2 x} &= \sum_{(m,n) \in [-K, K]^2} -4\pi^2 m^2\fourcoeff{m}{n}{1} e^{2\pi i mx}e^{2\pi i ny}, \\
        \frac{\partial^2 f(x, y)}{\partial x \partial y} &= \sum_{(m,n) \in [-K, K]^2} -4\pi^2 mn\fourcoeff{m}{n}{1} e^{2\pi i mx}e^{2\pi i ny}, \\
        \frac{\partial^2 f(x, y)}{\partial^2 y} &= \sum_{(m,n) \in [-K, K]^2} -4\pi^2 n^2\fourcoeff{m}{n}{1} e^{2\pi i mx}e^{2\pi i ny}.
    \end{align*}
    Observe that $\Expec{\left(\fourcoeff{m}{n}{j}\right)^2} = 0$ due to assumption (2) above, and so we can 
    compute the different variances as
    \begin{align*}
        \Var{\frac{\partial f(x, y)}{\partial x}} &= \sum_{(m,n) \in [-K, K]^2} 4\pi^2 m^2 = \frac{16}{3}\pi^2 K^4\left(1 + o\left(\frac{1}{\sqrt{K}}\right) \right) \\ &:= v_1(K) , \\
        \Cov{\frac{\partial f(x, y)}{\partial x}}{\frac{\partial f(x, y)}{\partial y}} &= \sum_{(m,n) \in [-K, K]^2} 4\pi^2 mn = 0 , \\
        \Var{\frac{\partial^2 f(x, y)}{\partial^2 x}} &= \sum_{(m,n) \in [-K, K]^2} 16\pi^4 m^4 = \frac{64}{5}\pi^4 K^6\left(1 + o\left(\frac{1}{\sqrt{K}}\right) \right) \\ &:= v_2(K) , \\
        \Var{\frac{\partial^2 f(x, y)}{\partial x \partial y}} &= \sum_{(m,n) \in [-K, K]^2} 16\pi^4 m^2n^2 = \frac{64}{9}\pi^4 K^6\left(1 + o\left(\frac{1}{\sqrt{K}}\right) \right)\\ &:= v_3(K) , \\
        \Cov{\frac{\partial^2 f(x, y)}{\partial^2 x}}{\frac{\partial^2 f(x, y)}{\partial^2 y}} &= \sum_{(m,n) \in [-K, K]^2} 16\pi^4 m^2n^2 = \frac{64}{9}\pi^4 K^6\left(1 + o\left(\frac{1}{\sqrt{K}}\right) \right) , \\
        \Cov{\frac{\partial^2 f(x, y)}{\partial x \partial y}}{\frac{\partial^2 f(x, y)}{\partial^2 y}} &= 0 , \\
        \Cov{\nabla^2 f}{\nabla f} &= 0.
    \end{align*}
    
    Assumptions \ref{ass:smooth}-\ref{ass:transverse} are clearly satisfied here.
    Note that $f$ is a stationary process as well here. So the formula \eqref{eq:nice avg len} reduces to
    \begin{equation}
        \label{eq:torus example}
    \Expec{\length(\Sigma_{c})} = \frac{\Expec{ \|\adj\left(\nabla^2f-\Expec{\nabla^2f | \nabla f}\right)\nabla f\| }}{2\sqrt{\Var{\nabla f(x)} }},
    \end{equation}
    since $n = 2$ and $\int_N dV = 1$ in this situation.
    
    Observe that $\nabla f$ is $\sqrt{v_1(K)}$ times a standard normal 2-vector, $\nabla^2 f$ is a random Gaussian symmetric matrix independent of $\nabla f$. Also note that $\frac{1}{(2K)^3\pi^2} \nabla^2 f$ is a random Gaussian symmetric matrix $M$ with
    \begin{align*}
        \Var{M_{ii}} = {\frac{1}{5} + o\left(\frac{1}{\sqrt{K}}\right)}, \Var{M_{ij}} = \Cov{M_{ij}}{M_{ii}} = {\frac{1}{9} + o\left(\frac{1}{\sqrt{K}}\right)}.
    \end{align*}

    Finally, we compute 
    \begin{align*}
        \sqrt{\det{\Var{\nabla f}}} =  \frac{16}{3}\pi^2 K^4\left(1 + o\left(\frac{1}{\sqrt{K}}\right) \right) = v_1(K) ,
    \end{align*}
    to apply \eqref{eq:torus example} and say that the average length of the critical curve is
    \begin{equation}
        \label{eq:bandlimited final}
        \frac{1}{2v_1(K)}\sqrt{v_1(K)}(2K)^3\pi^2l \approx \sqrt{3}\pi l K,
    \end{equation}
    where $l$ is a constant equal to $\Expec{\|\adj{M}Z\|}$ where $M$ is random Gaussian symmetric random matrix and $Z$ is an independent standard normal 2-vector with 
    \begin{align*}
        \Var{M_{ii}} = \frac{1}{5}, \Var{M_{ij}} = \Cov{M_{ij}}{M_{ii}} = \frac{1}{9}.
    \end{align*}

    \begin{figure}[ht!]
        \centering
        \begin{subfigure}{0.4\linewidth}
            \centering
                \includegraphics[width=0.7\textwidth]{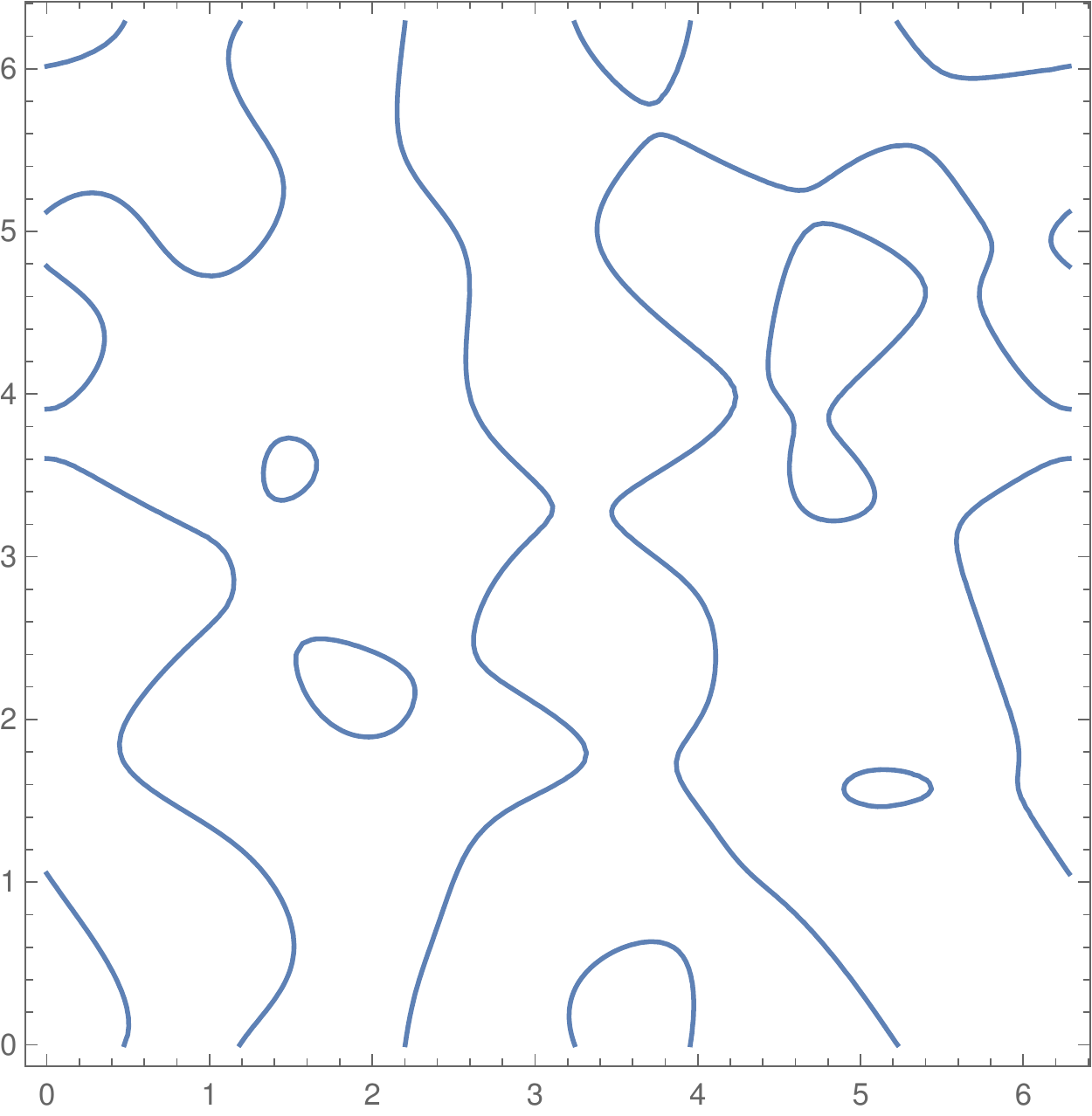}
                \caption{Bandwidth 2}
                \label{fig:p4a}
            \end{subfigure}
        \begin{subfigure}{0.4\linewidth}
            \centering
                \includegraphics[width=0.7\textwidth]{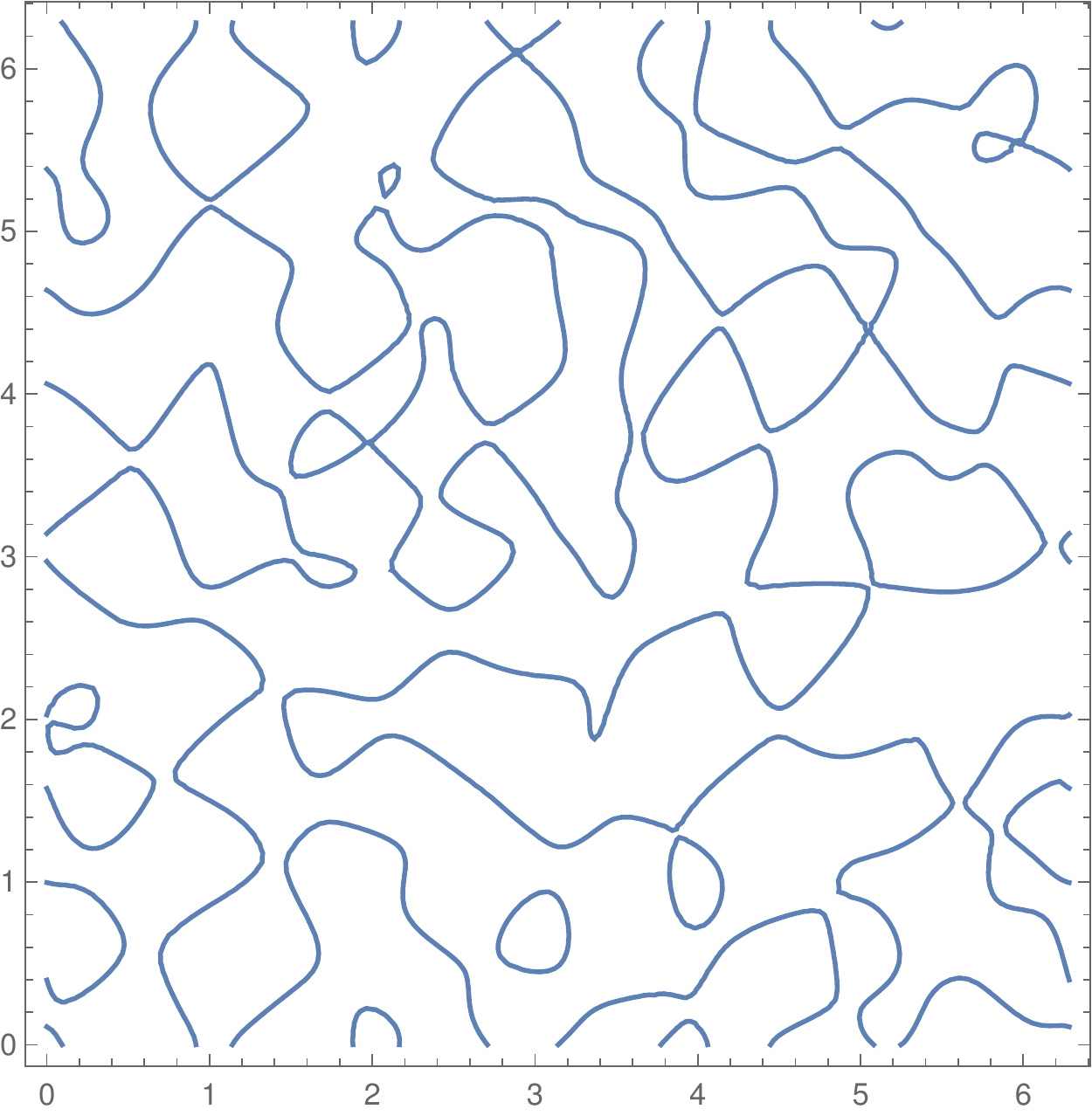}
                \caption{Bandwidth 4}
                \label{fig:p4a}
            \end{subfigure} \\
        \begin{subfigure}{0.4\linewidth}
            \centering
                \includegraphics[width=0.7\textwidth]{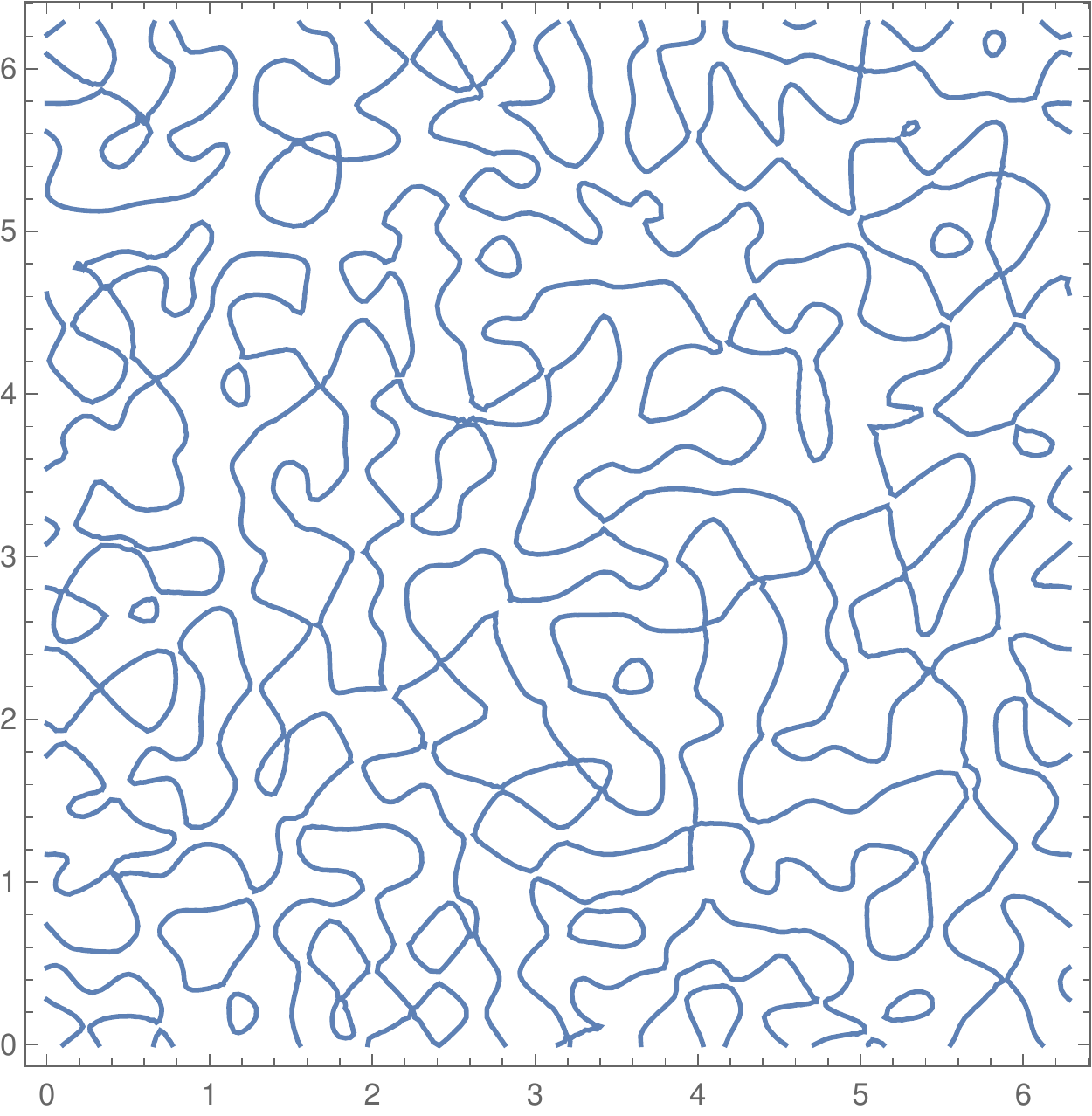}
                \caption{Bandwidth 8}
                \label{fig:p4a}
            \end{subfigure}
        \begin{subfigure}{0.4\linewidth}
            \centering
                \includegraphics[width=0.7\textwidth]{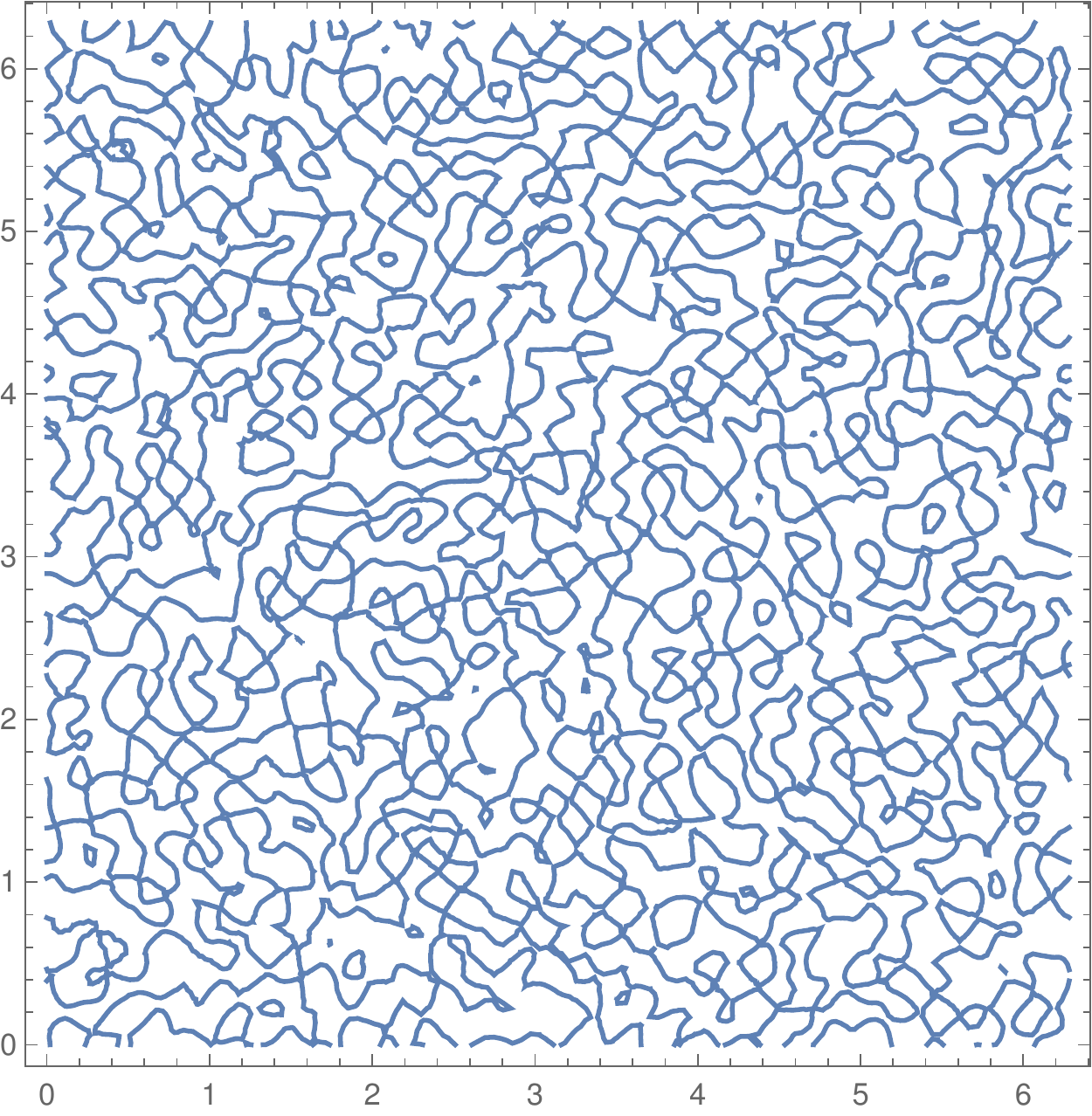}
                \caption{Bandwidth 16}
                \label{fig:p4a}
            \end{subfigure}
        \centering
        \caption{%
            Critical curves of bandlimited functions on the flat torus
        }%
        \label{fig:subfigures}
    \end{figure}
    We can also verify this linear relationship numerically in Mathematica. For each 
    value of $K$ in $\{2,3,...,10\}$, we choose 20 sets of Fourier coefficients drawn
    randomly according to the assumptions given in the beginning of this example.
    We then computed the sample average length of the critical curve for each $K$ and attach the plot
    in Figure \ref{fig:avg-len-stats}. The red points in the plot show the sample average
    critical curve lengths and the blue line is the best linear fit, which has almost
    zero y-intercept and is consistent with the observation in \eqref{eq:bandlimited final}.
    The slope of the best fit line is 3.33. We computed the constant $l$ approximately
    using a sample average with a large number of samples and found it is approximately
    0.607, which tells us that $\sqrt{3}\pi l \approx 3.3$. This is very close to the slope
    of the best fit line in Figure \ref{fig:avg-len-stats}.
    \begin{figure}[ht!]
        \begin{center}
            \includegraphics[width=8cm]{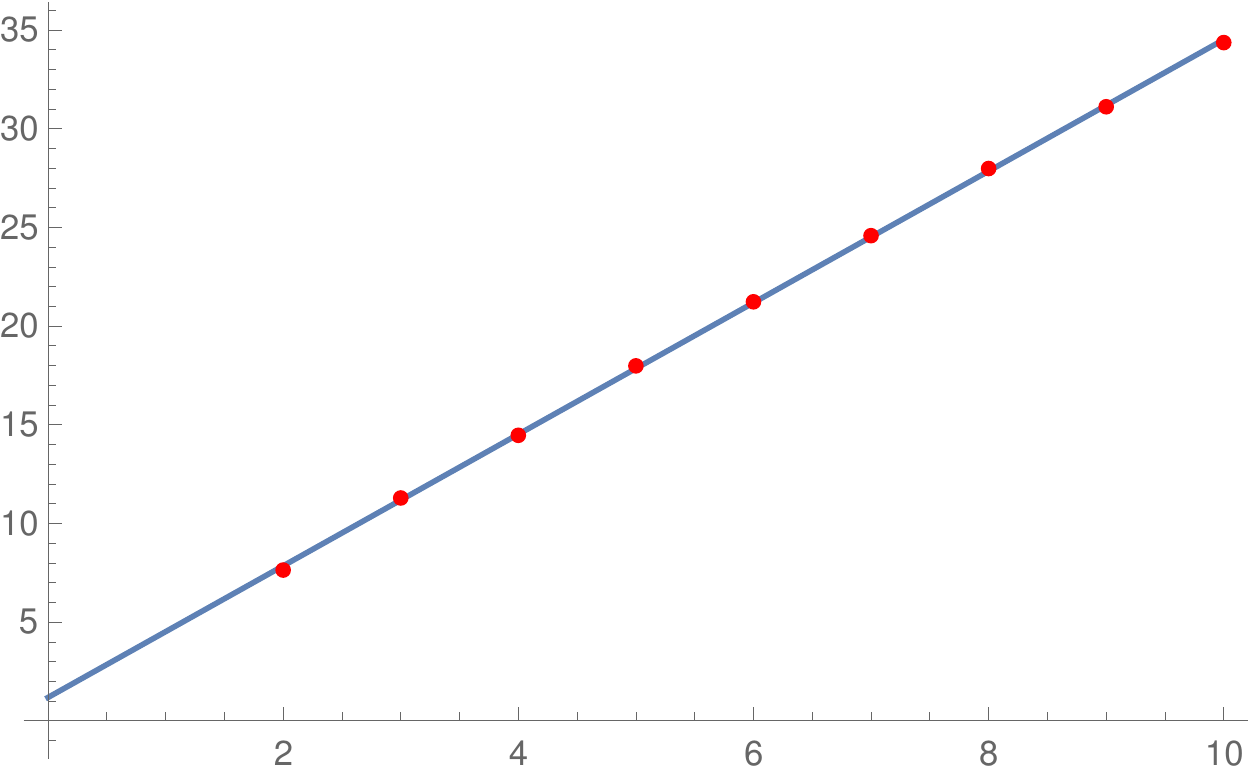}
        \end{center}
        \caption{%
            Average lengths of critical curves of bandlimited functions
        }%
        \label{fig:avg-len-stats}
    \end{figure}

    We can compute the average length of the visible contour in a similar fashion as
    \begin{equation}
        \label{eq:torus example 2}
        \Expec{\length(\gamma_c)} = \frac{\Expec{ \left|\nabla f^{\top}\adj\left(\nabla^2f-\Expec{\nabla^2f \big| \nabla f}\right)\nabla f\right| }}{2\sqrt{\Var{\nabla f(x)} }},
    \end{equation}
    which becomes
    \begin{equation}
        \label{eq:bandlimited final 2}
        \frac{1}{2v_1(K)}v_1(K)(2K)^3\pi^2c \approx 4\pi^2 c K^3,
    \end{equation}
    where $c = \Expec{\left|Z^{\top}MZ \right|}$.
\end{example}

\begin{example}[Random linear projections of a thin doughnut]
    In this example we consider $N$ to be the 2-Torus embedded in $\R^3$ in the shape
    of a hollow doughnut of radius $R$ and cross sectional radius $r$, and the 
    metric to be the metric induced from the Euclidean metric on $\R^3$. We denote
    by $T$ the ratio $\frac{R}{r}$. The embedding can be written in $\phi, \rho$ coordinates as
    \begin{align*}
        [0,2\pi]\times[0,2\pi] \ni (\phi, \rho) \mapsto E(\phi, \rho) = 
        \begin{bmatrix}
        \cos(\phi)\left(R+r\sin(\rho)\right) \\
        \sin(\phi)\left(R+r\sin(\rho)\right) \\
        r\cos(\rho))
    \end{bmatrix} \in \R^3.
    \end{align*}
    The induced metric tensor can then be written in $(\phi, \rho)$ coordinates 
    as
    \begin{align*}
        G(\phi, \rho) =
        \begin{bmatrix}
            (R+r\sin(\rho)) & 0 \\
            0 & r 
    \end{bmatrix} =
        \begin{bmatrix}
            r(T+\sin(\rho)) & 0 \\
            0 & r 
    \end{bmatrix}.
    \end{align*}
    We consider random linear projections of this embedding onto $\R^2$. That is,
    we choose $A \in \R^{2\times3}$ such that each entry is drawn from a standard
    Gaussian distribution independently of each other, and then define
    \begin{align*}
        h(\phi, \rho) := A.E(\phi, \rho) \in \R^2.
    \end{align*}
    The two components of $h$ are clearly identical and independent since the two rows
    of $A$ are i.i.d. and so equations \eqref{eq:nice avg len} and \eqref{eq:nice avg len 2}
    apply here. We can see that
    \begin{align*}
        f(\phi, \rho) = v^{\top}E(\phi, \rho)
    \end{align*}
    where $v$ is drawn from a standard 3D Gaussian distribution. The computations
    in this example are a bit tedious and are done in Mathematica.
    We can show that 
    \begin{align*}
        \Var{\nabla f(\phi, \rho)} = 
        \begin{bmatrix}
            (R+r\sin(\rho))^2 & 0 \\
            0 & r^2 
    \end{bmatrix} =
        \begin{bmatrix}
            r^2(T+\sin(\rho))^2 & 0 \\
            0 & r^2 
        \end{bmatrix},
    \end{align*}
    and 
    
    \begin{align*}
        &\Var{\begin{bmatrix}f_{\phi,\phi} \\
                f_{\phi,\rho} \\
            f_{\rho,\rho}\end{bmatrix} \bigg| \nabla f = 0 }(\phi, \rho) = 
\left[\begin{smallmatrix}
 \frac{1}{4} r^2 (\cos (4 \phi)+3) \sin ^2(\rho) (T+\sin (\rho))^2 \\
 \frac{1}{4} r^2 \sin (4 \phi) \sin ^2(\rho) \cos (\rho) (T+\sin (\rho)) \\
 \frac{1}{4} r^2 (\cos (4 \phi)+3) \sin ^3(\rho) (T+\sin (\rho)) 
\end{smallmatrix}\right. \\
&\left.\begin{smallmatrix}
        \frac{1}{4} r^2 \sin (4 \phi) \sin ^2(\rho) \cos (\rho) (T+\sin (\rho)) & \frac{1}{4} r^2 (\cos (4 \phi)+3) \sin ^3(\rho) (T+\sin (\rho)) \\
        \frac{1}{8} r^2 \sin ^2(2 \phi) \sin ^2(2 \rho) & \frac{1}{4} r^2 \sin (4 \phi) \sin ^3(\rho) \cos (\rho) \\
        \frac{1}{4} r^2 \sin (4 \phi) \sin ^3(\rho) \cos (\rho) & \frac{1}{32} r^2 \left((\cos (4 \phi)+7) (\cos (4 \rho)+3)+8 \sin ^2(2 \phi) \cos (2 \rho)\right)
\end{smallmatrix}\right].\\
            \end{align*}
    We compute the length of the visible contour when $T \gg 1$. To avoid cumbersome notation,
from this point on in this example we will denote $\nabla^2 f - \Expec{\nabla^2 f | \nabla f}$ as
just $\nabla^2 f$. We remind that the variance of this random symmetric matrix is given
in the previous equation, and it is independent of $\nabla f$.
    Observe that
                \[  \adj{\nabla^2 f} = \begin{bmatrix} f_{\rho, \rho} & -f_{\phi,\rho} \\
                -f_{\phi,\rho} & f_{\phi, \phi}\end{bmatrix}
    \]
    and
    \[
        \nabla f^{\top}(\adj{\nabla^2 f})\nabla f = f_{\phi}^2f_{\rho, \rho} + f_{\rho}^2f_{\phi,\phi} - 2f_{\phi}f_{\rho}f_{\phi,\rho}.
    \]
    In addition, $\sqrt{\det{\Var{\nabla f}(\phi,\rho)}} = r^2(T+\sin(\rho)).$
    This random field is not stationary, so we will need to integrate over the torus
    unlike the previous example. Equation \eqref{eq:nice avg len 2} gives us the
    average length of the visible contour as
    \begin{equation}
        \int_{[0,2\pi]}\int_{[0,2\pi]} \frac{\Expec{\left|f_{\phi}^2f_{\rho, \rho} + f_{\rho}^2f_{\phi,\phi} - 2f_{\phi}f_{\rho}f_{\phi,\rho}\right|}}{2r^2(T+\sin(\rho))} d\phi d\rho.
    \end{equation}
    Observe that
    \begin{align*}
        \Expec{|f_{\phi}^2f_{\rho,\rho}|} &= \Expec{|f_{\phi}^2}\Expec{|f_{\rho,\rho}|} = r^2(T+\sin(\rho))^2\sqrt{\frac{2}{\pi}}\sqrt{\Var{f_{\rho,\rho}}}, \\
        \Expec{|f_{\rho}^2f_{\phi,\phi}|} &= \Expec{|f_{\rho}^2}\Expec{|f_{\phi,\phi}|} = r^2\sqrt{\frac{2}{\pi}}\sqrt{\Var{f_{\phi,\phi}}}\\ &=  r^2(T+\sin(\rho))\sqrt{\frac{2}{\pi}}\sqrt{\frac{\Var{f_{\phi,\phi}}}{(T+\sin(\rho)^2}} ,\\
        \Expec{|f_{\phi}f_{\rho}f_{\rho,\phi}|} &= \Expec{|f_{\phi}|}\Expec{|f_{\rho}|}\Expec{|f_{\phi,\rho}|} = \frac{2}{\pi}r^2(T+\sin(\rho))\sqrt{\frac{2}{\pi}}\sqrt{\Var{f_{\phi,\rho}}}, \\
    \end{align*}
    where the expectations split as a product because $\nabla^2 f$ and $\nabla f$ 
    are independent, and we have used the fact that $\Expec{|X|} = \sqrt{\frac{2}{\pi}\Var{X}}$ 
    for a Gaussian random variable. Clearly, as $T \to \infty$ the $\Expec{|f_{\phi}^2f_{\rho,\rho}|}$ 
    term dominates and we can say that
    \begin{align*}
        \Expec{\length(\gamma_c)} &\approx \int_{[0,2\pi]}\int_{[0,2\pi]} \frac{\Expec{\left|f_{\phi}^2f_{\rho, \rho}\right|}}{2r^2(T+\sin(\rho))} d\phi d\rho \\
                                       &= \int_{[0,2\pi]}\int_{[0,2\pi]} \frac{(T+\sin(\rho))\sqrt{\Var{f_{\rho,\rho}}}}{\sqrt{2\pi}}  d\phi d\rho \\
                                       &= \frac{Tr}{\sqrt{2\pi}} \int_{[0,2\pi]}\int_{[0,2\pi]} (1+\frac{\sin(\rho)}{T})\sqrt{\Var{\frac{f_{\rho,\rho}}{r}}}  d\phi d\rho, \\
                                       &\approx \frac{Rc}{\sqrt{2\pi}} 
    \end{align*}
    where  
    \[\sqrt{\Var{\frac{f_{\rho,\rho}}{r}}} = 
    \sqrt{\frac{\left((\cos (4 \phi)+7) (\cos (4 \rho)+3)+8 \sin ^2(2 \phi) \cos (2 \rho)\right)}{32}} \]
    and 
    \[
        c = \int_{[0,2\pi]}\int_{[0,2\pi]} \sqrt{\Var{\frac{f_{\rho,\rho}}{r}}}  d\phi d\rho.
    \]
    The constant $c$ can be computed numerically as $\approx 31.6$ and so
    \begin{equation}
        \Expec{\length(\gamma_c)} \approx 12.6R, \quad \text{ when $\frac{R}{r} \gg 1$}.
    \end{equation}
    A similar computation will show that the length of the critical curve also grows
    linearly in $R$ when $\frac{R}{r} \gg 1$.
\end{example}

\section{Isotropic GRFs on Spheres}
\label{sec:iso}

All the formulae we computed in the previous section depends only on the
joint distribution of $(\nabla f(x), \nabla^2 f(x))$. When the space $N$ is
$\mathbb{S}^n$, and the GRF $f$ is isotropic and stationary, we will see that this joint distribution is particularly well structured.
Under these conditions, $\nabla f$ and $\nabla^2 f$ are independent, $\nabla f$ is a standard Gaussian random vector, and $\nabla^2 f$ is distributed as a Gaussian Orthogonally Invariant (GOI) ensemble. We will see in the following section 
some of the properties of GOI ensembles that will lead to more reduced formulae for the various computations
we did in earlier sections. Most of the results about GOI ensembles mentioned
here are a review of what can be found in \cite{Che18}.

\subsection{Gaussian Orthogonally Invariant Ensembles}
An $n \times n$ random matrix $H_{ij}$ is said to have Gaussian Orthogonal Ensemble (GOE) distribution if it is symmetric and all entries
are centered Gaussian random variables with 
\[
    \Expec{H_{ij}H_{kl}} = \frac{1}{2}(\delta_{ik}\delta_{jl}+\delta_{il}\delta_{jk} ).
\]
It is well known that the GOE ensemble is orthogonally invariant i.e. the distribution of $H$ is the same as that of $QHQ^\top$ for any orthogonal matrix $Q$. Moreover, the entries of $H$ are independent. However, we will need a slightly more general distribution to capture the structure of the Hessian of isotropic GRFs.

An $n \times n$ random matrix $M_{ij}$ is said to have Gaussian Orthogonally Invariant distribution with covariance parameter $c$ (GOI(c)) if it is symmetric and all entries
are centered Gaussian random variables with 
\[
    \Expec{M_{ij}M_{kl}} = \frac{1}{2}(\delta_{ik}\delta_{jl}+\delta_{il}\delta_{jk} + c\delta_{ij}\delta_{kl} ).
\]
The GOI distribution is also orthogonally invariant. In fact, upto a scaling constant any orthogonally invariant symmetric Gaussian random matrix has to have GOI(c) distribution. The only constraint on the covariance parameter is that $c \geq -1/N$ (\cite[Lemma 2.1]{Che18}). We will see that the Hessian of isotropic GRFs are GOI ensembles. The orthogonal invariance of GOI random matrices will prove a very useful property in our computations. In addition, the density of the ordered eigenvalues of GOI(c) matrices can be written as
\begin{equation}
\begin{aligned}
    f_c(\lambda_1,...,\lambda_n) &= \frac{1}{K_n\sqrt{1+nc}}\exp\left\{ -\frac{1}{2}\sum_{i=1}^{n} \lambda_i^2 + \frac{c}{2(1+nc)}\left(\sum_{i=1}^{n} \lambda_i \right)^2  \right\} \\
                                 & \times \prod_{1\leq i < j \leq n} |\lambda_i-\lambda_j|\indic{\{\lambda_1\leq...\leq\lambda_n\}}
\end{aligned}
\end{equation}
where $K_n$ is the normalization constant
\[
    K_n = 2^{n/2} \prod_{i=1}^{n} \Gamma\left(\frac{i}{2}\right).
\]
For any measurable function $g$, we will denote by
\[
    \ExpecGOI{g(\lambda_1,...,\lambda_n)} := \int_{\R^n} g(\lambda_1,...,\lambda_n)f_c(\lambda_1,...,\lambda_n)d\lambda_1...d\lambda_n
\]
the expectation under GOI(c) density.

\subsection{Length computations on Isotropic GRFs on spheres}

Let $\mathbb{S}^n = \left\{(z_1,..,z_{n+1}) \in \R^{n+1} \bigg| \sum\limits_{i=1}^{n+1} z_i^2 = 1\right\}$ be the unit $n$-sphere embedded in $\R^{n+1}$ endowed
with the induced Riemannian metric, and $f$ be a centered, unit-variance, smooth isotropic GRF on $\mathbb{S}^n$.
Due to isotropy, we can write the covariance function $R$ of $f$ as $R(x,y) = C(\langle x,y \rangle)$
for some $C:[-1,1] \rightarrow \R$. Define
\[
    C' = C'(1), C'' = C''(1), \eta = \sqrt{C'}/\sqrt{C''}, \kappa = C'/\sqrt{C''}. 
\]
Since an isotropic GRF is also stationary, we only need to compute the integrands in equations \eqref{eq:nice avg len}-\eqref{eq:nice avg len 2} at one point on the sphere. We will choose this point to be the 
north pole $N = (0,...,0,1)$, and use the fact that $(z_1,...,z_n)$ forms a coordinate chart on the sphere in a neighborhood around this point. In this coordinate system, the metric tensor $G(N)$ is simply the identity matrix $I_n$. In addition, we have the following lemma giving us the distribution of the derivatives of $f$.

\begin{lem}{\cite[Lemma 4.1, 4.3]{Che18}}
    \label{l:iso lem}
    Let $f$ be a centered, unit-variance, smooth isotropic GRF on $\mathbb{S}^n$. Then
    \begin{enumerate}
        \item $\nabla f$ is $\sqrt{C'}$ times a standard Gaussian random vector,
        \item $\nabla^2 f$ is $\sqrt{2C''}$ times a GOI($(1+\eta^2)/2$) matrix,
            \item $\nabla f$ and $\nabla^2 f$ are independent,
    \end{enumerate}
    where the derivatives are computed in the $(z_1,...,z_n)$ coordinates at $N$.
\end{lem}

We then have the following result giving a nicer formula for the expected length 
of the visible contour,

\begin{thm} If $ N = \mathbb{S}^n$ and the components of the GRF $f$ and $g$ are independent and identically
    distributed as centered, unit-variance, smooth isotropic GRFs with $C^{'}\neq 0, C^{''} \neq 0$,
    then the expected value of the length of its visible contour
    of index $k$ is
\label{th:avg len ccur ind iso}
\begin{equation}
    \Expec{\length{\gamma_c^k}} = \frac{\sqrt{2\pi^3}n}{\Gamma(\frac{n+1}{2})}\frac{\kappa}{\eta^n} \ExpecGOIp{\prod_{i=1}^{n-1}|\lambda_i| \indic{\lambda_k < 0 < \lambda_{k+1}}}{\frac{1+\eta^2}{2}}{n-1}.
\end{equation}
\end{thm}
\begin{proof}
    Assumption \ref{ass:transverse} is satisfied here, since $(f, \nabla f)$ are 
    independent, $f$ has unit-variance and lemma \ref{l:iso lem} implies $\nabla f$ is
    non degenerate. Equation \eqref{eq:nice avg len 2} gives
\begin{align*}
    \Expec{\length{\gamma_c}} &= \vol{(\mathbb{S}^n)}. \frac{\pi C' (\sqrt{2C''})^{n-1}}{(\sqrt{2\pi C'})^n}\Expec{\left|v^{\top}\adj(M)v\right|} \\
                              &= \frac{2 \sqrt{\pi^{n+1}}}{\Gamma(\frac{n+1}{2})}\frac{\pi}{\sqrt{2}\sqrt{\pi^{n}}} {\frac{(\sqrt{C''})^{n-1}}{(\sqrt{C'})^{n-2}}} \Expec{\left|v^{\top}\adj(M)v\right|} \\
                              &= \frac{\sqrt{2\pi^3}}{\Gamma(\frac{n+1}{2})}\frac{\kappa}{\eta^n}\Expec{\left|v^{\top}\adj(M)v\right|},
\end{align*}
where $v$ is a standard unit Gaussian random vector independent of the GOI($\frac{1+\eta^2}{2}$) matrix $M$.
Since $Q^{\top}\adj(M)Q = \adj(Q^{\top}MQ)$ for any orthogonal matrix $Q$, 
\begin{align*}
    \Expec{\left|v^{\top}\adj(M)v\right|} &= \Expec{\|v\|^2 \left|(e_1)^{\top}\adj(R^{\top}MR)e_1\right|} \\
                             &= n \Expec{\left| \adj(M)_{11}  \right| } \\
                             &= n \ExpecGOIp{\prod_{i=1}^{n-1}|\lambda_i|}{\frac{1+\eta^2}{2}}{n-1},
\end{align*}
which gives
\begin{align*}
    \Expec{\length{\gamma_c}} = \frac{\sqrt{2\pi^3}n}{\Gamma(\frac{n+1}{2})}\frac{\kappa}{\eta^n} \ExpecGOIp{\prod_{i=1}^{n-1}|\lambda_i|}{\frac{1+\eta^2}{2}}{n-1}.
\end{align*}

We can similarly compute the expected length of the critical curve of index $k$ 
using \eqref{eq:nice avg len 2} as
\begin{align*}
    \Expec{\length{\gamma_c^k}} &= \frac{\sqrt{2\pi^3}}{\Gamma(\frac{n+1}{2})}\frac{\kappa}{\eta^n}\Expec{\left|v^{\top}\adj(M)v\right| \indic{\ind(M)-\indic{v^\top M^{-1} v < 0 } = k} } \\
                                &= \frac{\sqrt{2\pi^3}}{\Gamma(\frac{n+1}{2})}\frac{\kappa}{\eta^n}n\Expec{\left|e_1^{\top}\adj(Q^\top MQ)e_1\right| \indic{\ind(Q^\top MQ)-\indic{e_1^\top (Q^\top MQ)^{-1} e_1 < 0 } = k} } \\
                                &= \frac{\sqrt{2\pi^3}n}{\Gamma(\frac{n+1}{2})}\frac{\kappa}{\eta^n}\Expec{\left|e_1^{\top}\adj(M)e_1\right| \indic{\ind(M)-\indic{e_1^\top (M)^{-1} e_1 < 0 } = k} }.
\end{align*}
If we denote by $\hat{M}$ the first $n-1 \times n-1$ principal minor of $M$, we
can reduce the above expectation as
\begin{align*}
    \mathsf{E} \bigg[ &\left|\det(\hat{M})\right|\left( \indic{\ind(M)=k, (-1)^k\det(\hat{M}) > 0} + \indic{\ind(M)=k+1, (-1)^{k+1}\det(\hat{M}) < 0 }  \right) \bigg].
\end{align*}
If the index of $M$ is $k$, then the index of $\hat{M}$ can either be $k$ or $k-1$. However,
$(-1)^k\det(\hat{M}) > 0$ implies that the index of $\hat{M}$ is $k$. Similarly, if the index 
of $M$ is $k+1$, then the index of $\hat{M}$ has to be either $k+1$ or $k$, but
$(-1)^{k+1}\det(\hat{M}) < 0$ implies that the the index of $\hat{M}$ is $k$. This allows
us to further reduce the above expectation as

\begin{align*}
    &\Expec{\left|\det(\hat{M})\right|\left( \indic{\ind(M)=k, \ind(\hat{M})=k} + \indic{\ind(M)=k+1, \ind(\hat{M})=k}  \right)} \\
    &=\Expec{\left|\det(\hat{M})\right| \indic{\ind(\hat{M})=k}}.
\end{align*}
Thus, we can write
\begin{align*}
    \Expec{\length{\gamma_c^k}} &= \frac{\sqrt{2\pi^3}n}{\Gamma(\frac{n+1}{2})}\frac{\kappa}{\eta^n}\Expec{\left|\det(\hat{M})\right|\indic{\ind(\hat{M})=k} } \\
                                &= \frac{\sqrt{2\pi^3}n}{\Gamma(\frac{n+1}{2})}\frac{\kappa}{\eta^n} \ExpecGOIp{\prod_{i=1}^{n-1}|\lambda_i| \indic{\lambda_k < 0 < \lambda_{k+1}}}{\frac{1+\eta^2}{2}}{n-1}.
\end{align*}
\end{proof}

\section{Expected number of pseudocusps}

Pseudocusps can be split into two obvious types, vertical and horizontal,
depending on whether they're the image of critical points of $f$ or $g$
respectively. Recall the notation from section \ref{sec:exp-length}; if a
point $p$ on $N$ is a critical point of $f$, one can locally parametrize the critical curve
in its neighborhood as a function $x(\theta)$ of $\theta$ such that $x(0) = p$.
Recall that if 
\[
    V(x, \theta) = \cos(\theta) \nabla f(x) + \sin(\theta) \nabla g(x) = 0
\]
then the tangent line to the visible contour at $h(x)$ lies along $(-\sin(\theta), \cos(\theta))$.
This means the slope of the visible contour at the image of $x(\theta)$ is just
$\theta - \frac{\pi}{2}$. For the image of a critical point $p$ to be a vertical
pseudocusp, the direction of the tangent vector to the visible contour along the direction
of increasing slope must point upward, which means
\[
    \left\langle \frac{d h(x(\theta))}{d\theta} \bigg|_{\theta = 0}, \ \begin{bmatrix} 0 \\ 1 \end{bmatrix} \right\rangle > 0 \implies 
    -\nabla g(p) ^\top \left( \nabla^2 f(p) \right)^{-1} \nabla g(p) > 0.
\]

The index of the extension ray attached to a vertical pseudocusp, i.e. the dimension
of the cell attached to the sublevel set when the parameter value crosses a point
on the extension ray, is just the index of $\nabla^2 f(p)$. We call this the index
of a vertical pseudocusp. The same definition with $g$ replacing $f$ holds for
horizontal pseudocusps. Therefore, a vertical pseudocusp of index $k$ is characterized by the conditions
\begin{equation}
    \label{eq:vpc}
    \nabla f(p) = 0, \
    \nabla g(p) ^\top \left( \nabla^2 f(p) \right)^{-1} \nabla g(p) < 0, \
    \ind\left(\nabla^2 f(p)\right) = k.
\end{equation}
while a horizontal pseudocusp of index $k$ is characterized by
\begin{equation}
    \label{eq:hpc}
    \nabla g(p) = 0, \
    \nabla f(p) ^\top \left( \nabla^2 g(p) \right)^{-1} \nabla f(p) < 0, \
    \ind\left(\nabla^2 g(p)\right) = k.
\end{equation}

If $N_{vpc}^k$ and $N_{hpc}^k$ denote the number of vertical and horizontal pseudocusps of index
$k$, the next theorem gives a formula for the expected number of these points.

\begin{thm} If the GRF $h$ satisfies assumptions \ref{ass:smooth}-\ref{ass:transverse},
       then the expected number of pseudocusps of index $k$ are
\label{th:pseudocusps}
\begin{equation}
    \label{eq:vpc}
    \Expec{N_{vpc}^k} = \\ \int_N \frac{\Expec{ \left|\det\left(\nabla^2 f(p)\right) \right|\indic{\ind\left(\nabla^2 f(p)\right) = k, \ \nabla g(p)^\top \left( \nabla^2 f(p) \right)^{-1} \nabla g(p) < 0 } \bigg| \nabla f(p) = 0}}{\sqrt{(2\pi)^n\det\left(\Var{\nabla f(p)} \right)}} dV,
\end{equation}
and
\begin{equation}
    \label{eq:hpc}
    \Expec{N_{hpc}^k} = \\ \int_N \frac{\Expec{ \left|\det\left(\nabla^2 g(p)\right) \right|\indic{\ind\left(\nabla^2 g(p)\right) = k, \ \nabla f(p)^\top \left( \nabla^2 g(p) \right)^{-1} \nabla f(p) < 0 } \bigg| \nabla g(p) = 0}}{\sqrt{(2\pi)^n\det\left(\Var{\nabla g(p)} \right)}} dV,
\end{equation}
\end{thm}

\begin{proof}
    This is a consequence of the characterizations \eqref{eq:vpc}, \eqref{eq:hpc} 
    and the Kac-Rice formula \cite[Theorem 12.1.1]{AdlTay07}.
\end{proof}

\section{Conclusion}

We have computed the expected length of the critical curve and visible contour of fixed
index of a smooth centered Gaussian random map into the plane in this article. We derived
more explicit expressions in the case where the components are identical and independent
and a closed form expression under the additional assumption of isotropy. We also
computed the expected number of pseudocusps of such a Gaussian random map.

The one remaining singularity appearing in the description of biparametric persistence
is the cusp point; these are the points where the visible contour loses smoothness.
We have not treated these singularities in this article. The cusps points can
be characterized as points where the 2-jet $j^2 h$ intersects a certain submanifold
$S_{1,1} \subset J^2(N, \R^2)$ of codimension $n$. If the support of $j^2 h(p)$ is full
for all $p \in N$, these intersections will be transverse at all $p \in N$ almost surely.
We can then compute the expected number of these cusps as the number of transverse
intersections of the function $j^2 h$ with the submanifold $S_{1,1}$ using a
generalized Kac-Rice formula \cite{Ste21}. However, these computations are a bit 
cumbersome and we will pursue these in a future work.

\bibliographystyle{alpha}
\bibliography{references}
\end{document}